\documentclass{amsart}%
\usepackage{amssymb}
\usepackage{amsfonts}
\usepackage{geometry}
\usepackage{graphicx}%
\usepackage{amsmath}%
\providecommand{\U}[1]{\protect\rule{.1in}{.1in}}
\newtheorem{theorem}{Theorem}
\theoremstyle{plain}

\newtheorem{conjecture}[theorem]{Conjecture}

\newtheorem{definition}[theorem]{Definition}

\newtheorem{lemma}[theorem]{Lemma}
\newtheorem{notation}[theorem]{Notation}

\newtheorem{proposition}[theorem]{Proposition}

\numberwithin{equation}{section}
\begin{document}
\title[Comparison of testing conditions]{A comparison of trilinear testing conditions for the paraboloid Fourier
extension and Kakeya conjectures in three dimensions}
\author{Eric Sawyer}
\address{McMaster University\\
Hamilton, Ontario, Canada}
\email{sawyer@mcmaster.ca}
\thanks{Eric Sawyer's research supported in part by a grant from the National Sciences
and Engineering Research Council of Canada}
\maketitle

\begin{abstract}
We compare the smooth Alpert testing condition for the paraboloid Fourier
extension conjecture in \cite{RiSa3} to the modulated testing condition for
the Kakeya conjecture in \cite{RiSa2}. To this end, the modulated testing
condition is converted to a certain restricted smooth Alpert testing condition
for the paraboloid Fourier extension conjecture.

\end{abstract}
\tableofcontents

\section{Introduction}

The purpose of this paper is to compare the trilinear characterizations of the
paraboloid Fourier extension conjecture and the Kakeya strong maximal operator
conjecture appearing in \cite{RiSa3} and \cite{RiSa2} respectively, and then
to discuss some consequences in the final subsection of this introduction. We
begin by recalling certain of the trilinear characterizations in
$\mathbb{R}^{3}$ obtained in \cite{RiSa3} and \cite{RiSa2}.

\subsection{Paraboloid Fourier extension conjecture in $\mathbb{R}^{3}$}

The Fourier extension operator $\mathcal{E}$ on the paraboloid in three
dimensions is given by,%
\[
\mathcal{E}f\left(  \xi\right)  \equiv\left[  \Phi_{\ast}\left(  f\left(
x\right)  dx\right)  \right]  ^{\wedge}\left(  \xi\right)  =\int_{U}%
e^{-i\Phi\left(  x\right)  \cdot\xi}f\left(  x\right)  dx,\ \ \ \ \ \text{for
}\xi\in\mathbb{R}^{3},
\]
where $\Phi_{\ast}\left(  f\left(  x\right)  dx\right)  $ denotes the
pushforward of the measure $f\left(  x\right)  dx$ supported in $U\subset
B_{\mathbb{R}^{2}}\left(  0,\frac{1}{2}\right)  $ to the paraboloid
$\mathbb{P}^{2}$ under the usual parameterization $\Phi:U\rightarrow
\mathbb{P}^{2}$ by $\Phi\left(  x\right)  =\left(  x_{1},x_{2},x_{1}^{2}%
+x_{2}^{2}\right)  $ for $x=\left(  x_{1},x_{2}\right)  \in U$, where the
pushforward $\Phi_{\ast}\mu$ of a measure $\mu$ is defined by the identity
$\left\langle g,\Phi_{\ast}\mu\right\rangle =\left\langle g\circ\Phi
,\mu\right\rangle $ for all $g\in C_{c}\left(  \mathbb{R}^{3}\right)  $. We
will often abuse notation and simply write $\Phi_{\ast}f$ where we view $f$ as
the measure $f\left(  x\right)  dx$. The Fourier extension conjecture for the
paraboloid $\mathbb{P}^{2}$ in $\mathbb{R}^{3}$ is,%
\begin{equation}
\left\Vert \mathcal{E}f\right\Vert _{L^{q}\left(  \mathbb{R}^{3}\right)  }\leq
C\left\Vert f\right\Vert _{L^{q}\left(  U\right)  },\ \ \ \ \ \text{for all
}q>3\text{ and }f\in L^{q}\left(  U\right)  .\label{FEC}%
\end{equation}
We will often refer to this inequality as the \emph{linear} Fourier extension
conjecture to distinguish it from the trilinear variants below.

Define
\[
\mathsf{Q}_{s,U}f=\sum_{I\in\mathcal{G}_{s}\left[  U\right]  }\bigtriangleup
_{I;\kappa}^{\eta}f=\sum_{I\in\mathcal{G}_{s}\left[  U\right]  }\left\langle
f,h_{I;\kappa}\right\rangle h_{I;\kappa}^{\eta}%
\]
to be the smooth Alpert pseudoprojection of $f$ at level $s$ in the grid
$\mathcal{G}$ localized to $U$, where $\mathcal{G}$ is as in \cite{RiSa3}.
Here $h_{I;\kappa}^{\eta}$ is a smooth Alpert wavelet with $\kappa$ vanishing
moments as introduced in \cite{Saw7}, and also described below. We say that a
triple $\left(  U_{1},U_{2},U_{3}\right)  $ of squares $U_{k}\subset U$ is
$\nu$\emph{-disjoint} if%
\begin{equation}
\ell\left(  U_{k}\right)  \approx\nu,\text{ and }\operatorname*{dist}\left(
U_{j},U_{k}\right)  \geq\nu,\ \ \ \ \ 1\leq j,k\leq3.\label{nu disjoint'}%
\end{equation}

Now we can describe the smooth Alpert trilinear characterization from
\cite{RiSa3}.

\begin{definition}
Let $1<q<\infty$ and $0<\varepsilon,\delta,\nu<1$ and $\kappa\in\mathbb{N}$.
We say the single scale disjoint\emph{\ }trilinear inequality $\mathcal{E}%
_{\operatorname*{disj}\nu}^{\delta,\kappa}\left(  \otimes_{3}L^{\infty
}\rightarrow L^{\frac{q}{3}};\varepsilon\right)  $ holds if there is a
positive constant $C_{q,\delta,\kappa,\varepsilon,\nu}$ such that%
\[
\left\Vert \mathcal{E}\mathsf{Q}_{s,U_{1}}^{\eta}f_{1}\ \mathcal{E}%
\mathsf{Q}_{s,U_{2}}^{\eta}f_{2}\ \mathcal{E}\mathsf{Q}_{s,U_{3}}^{\eta}%
f_{3}\right\Vert _{L^{\frac{q}{3}}\left(  B_{\mathbb{R}^{3}}\left(
0,2^{\frac{s}{1-\delta}}\right)  \right)  }\leq C_{q,\delta,\kappa
,\varepsilon,\nu}2^{\varepsilon s}\left\Vert f_{1}\right\Vert _{L^{\infty
}\left(  U_{1}\right)  }\left\Vert f_{2}\right\Vert _{L^{\infty}\left(
U_{2}\right)  }\left\Vert f_{3}\right\Vert _{L^{\infty}\left(  U_{3}\right)
}\ ,
\]
holds for all $s\in\mathbb{N}$, $\varepsilon>0$, all single scale smooth
Alpert pseudoprojections $\mathsf{Q}_{s,U_{k}}f_{k}=\sum_{I\in\mathcal{G}%
_{s}\left[  U_{k}\right]  }\bigtriangleup_{I;\kappa}^{\eta}f$ with $\kappa$
vanishing moments, and all $\nu$-disjoint triples $\left(  U_{1},U_{2}%
,U_{3}\right)  \in U^{3}$.
\end{definition}

\begin{conjecture}
\label{tFec}Suppose $0<\delta<1$ and $\kappa\in\mathbb{N}$ with $\kappa
>\frac{20}{\delta}$. Then the smooth $\left(  \delta,\kappa\right)  $-Alpert
trilinear Fourier extension conjecture for the paraboloid $\mathbb{P}^{2}$ in
$\mathbb{R}^{3}$ is that for every $q>3$ there is $\nu>0$ depending on $q $,
such that the single scale disjoint\emph{\ }trilinear inequality
$\mathcal{E}_{\operatorname*{disj}\nu}^{\delta,\kappa}\left(  \otimes
_{3}L^{\infty}\rightarrow L^{\frac{q}{3}};\varepsilon\right)  $ holds.
\end{conjecture}

\begin{theorem}
[C. Rios and E. Sawyer \cite{RiSa2}]\label{single main Alpert}Fix $0<\delta<1$
and $\kappa\in\mathbb{N}$ with $\kappa>\frac{20}{\delta}$. Then the linear
Fourier extension conjecture (\ref{FEC}) for the paraboloid $\mathbb{P}^{2}$
in $\mathbb{R}^{3}$ holds \emph{if and only if} the $\left(  \delta
,\kappa\right)  $-Alpert trilinear Fourier extension Conjecture \ref{tFec} holds.
\end{theorem}

This theorem says that the linear Fourier extension conjecture is equivalent
to a corresponding disjoint trilinear smooth Alpert conjecture at single
scales, that permits the mild growth constant $2^{\varepsilon s}$ and at the
expense of taking integration over the local ball $B_{\mathbb{R}^{3}}\left(
0,2^{\frac{s}{1-\delta}}\right)  $ (larger than $B_{\mathbb{R}^{3}}\left(
0,2^{s}\right)  $), where $s$ is the scale of the smooth Alpert pseudoprojection.

\subsection{Kakeya conjecture in $\mathbb{R}^{3}$}

Let $0<\varepsilon<1$. We say the statement $\mathcal{K}^{\ast}\left(
\otimes_{1}L^{\infty}\rightarrow L^{\frac{3}{2}};\varepsilon\right)  $ holds
if there is a positive constant $C_{\varepsilon}$ such that%
\begin{align}
& \left\Vert \sum_{T\in\mathbb{T}}\mathbf{1}_{T}\right\Vert _{L^{\frac{3}{2}%
}\left(  \mathbb{R}^{3}\right)  }\leq C_{\varepsilon}\delta^{-\varepsilon
},\label{Kak dual}\\
& \text{for all families }\mathbb{T}\text{ of }\delta\text{-separated }%
\delta\text{-tubes in }\mathbb{R}^{n}\text{ and }0<\delta<1.\nonumber
\end{align}

\begin{conjecture}
[strong Kakeya maximal operator conjecture]\label{linear Kakeya}The statement
$\mathcal{K}^{\ast}\left(  \otimes_{1}L^{\infty}\rightarrow L^{\frac{3}{2}%
};\varepsilon\right)  $ holds for all $0<\varepsilon<1$.
\end{conjecture}

For a fixed grid $\mathcal{G}$ in $\mathbb{R}^{2}$ and $s\in\mathbb{N}$ define%
\[
\mathcal{G}_{s}\left[  U\right]  \equiv\left\{  I\in\mathcal{G}:I\subset
U\text{ and }\ell\left(  I\right)  =2^{-s}\right\}  .
\]
For each $s\in\mathbb{N}$, we fix a $2^{-s}$-separated subset $\mathcal{G}%
_{s}^{\ast}\left[  U\right]  $ of $\mathcal{G}_{s}\left[  U\right]  $. As this
sequence of squares remains fixed throughout the paper - up until the very end
- we will simply write $\mathcal{G}_{s}\left[  U\right]  $, with the
understanding that the squares $I\in\mathcal{G}_{s}\left[  U\right]  $ are
$2^{-s}$-separated. In any event, $\mathcal{G}_{s}\left[  U\right]  $ is a
finite union of collections of the form $\mathcal{G}_{s}^{\ast}\left[
U\right]  $.

\begin{definition}
\label{def M}Set $\mathcal{V}_{s}\equiv\left\{  \mathbb{R}^{3}\text{-valued
sequences on }\mathcal{G}_{s}\left[  U\right]  \right\}  $. For $s\in
\mathbb{N} $ and $\mathbf{u}=\left\{  u_{I}\right\}  _{I\in\mathcal{G}%
_{s}\left[  U\right]  }\in\mathcal{V}_{s}$, define the modulation
$\mathsf{M}_{\mathbf{u}}^{s}$ on $\Phi\left(  U\right)  $ by,%
\[
\mathsf{M}_{\mathbf{u}}^{s}\left(  z\right)  \equiv\sum_{I\in\mathcal{G}%
_{s}\left[  U\right]  }e^{iu_{I}\cdot z}\mathbf{1}_{\Phi\left(  2I\right)
}\left(  z\right)  ,\ \ \ \ \ \text{\ for }z\in\Phi\left(  U\right)  .
\]

\end{definition}

\begin{definition}
\label{def Q}Let $I_{0}\equiv\left[  0,1\right]  ^{2}$ be the unit square in
the plane. Fix $\varphi\in C_{c}^{\infty}\left(  2I_{0}\right)  $ such that
$\varphi=1$ on $I_{0}$. Then for any square $I$, let $\varphi_{I}$ be the
$L^{2}$ normalized translation and dilation of $\varphi$ that is adapted to
$I$. Given $f\in L^{1}\left(  U\right)  $ and $s\in\mathbb{N}$ define the
pseudoprojections
\[
\bigtriangleup_{I}^{\varphi}f\equiv\left\langle f,\varphi_{I}\right\rangle
\varphi_{I}\text{ and }\mathsf{Q}_{s,U}^{\varphi}f\equiv\sum_{I\in
\mathcal{G}_{s}\left[  U\right]  }\bigtriangleup_{I}^{\varphi}f=\sum
_{I\in\mathcal{G}_{s}\left[  U\right]  }\left\langle f,\varphi_{I}%
\right\rangle \varphi_{I}\ .
\]

\end{definition}

Note that the assumption on separation of squares in $\mathcal{G}_{s}\left[
U\right]  =\mathcal{G}_{s}^{\ast}\left[  U\right]  $ implies that $\left\{
\bigtriangleup_{I}^{\varphi}f\right\}  _{I\in\mathcal{G}_{s}\left[  U\right]
}$ is a collection of orthogonal pseudoprojections at level $s$, by which we
mean that for $I,L\in\mathcal{G}_{s}\left[  U\right]  =\mathcal{G}_{s}^{\ast
}\left[  U\right]  $,%
\begin{equation}
\bigtriangleup_{L}^{\varphi}\bigtriangleup_{I}^{\varphi}f=\left\{
\begin{array}
[c]{ccc}%
c_{\flat}\bigtriangleup_{I}^{\varphi}f & \text{ if } & L=I\\
0 & \text{ if } & L\not =I
\end{array}
\right.  ,\ \ \ \ \ \text{where }c_{\flat}=\left\langle \varphi_{I}%
,\varphi_{I}\right\rangle =\left\langle \varphi,\varphi\right\rangle
\approx1.\label{pseudo}%
\end{equation}
Now we recall the single scale modulated Fourier square function used in
\cite{RiSa2}.

\begin{definition}
For $s\in\mathbb{N}$ and $\mathbf{u}\in\mathcal{V}_{s}$ define the
\emph{modulated\ single scale }Fourier square function by
\begin{equation}
\mathcal{S}_{\operatorname*{Fourier}}^{s,\varphi,\mathbf{u}}f\left(
\xi\right)  \equiv\left(  \sum_{I\in\mathcal{G}_{s}\left[  U\right]
}\left\vert \left(  \mathsf{M}_{\mathbf{u}}^{s}\Phi_{\ast}\bigtriangleup
_{I}^{\varphi}f\right)  ^{\wedge}\left(  \xi\right)  \right\vert ^{2}\right)
^{\frac{1}{2}}=\left(  \sum_{I\in\mathcal{G}_{s}\left[  U\right]  }\left\vert
\tau_{u_{I}}\widehat{\Phi_{\ast}\bigtriangleup_{I}^{\varphi}f}\left(
\xi\right)  \right\vert ^{2}\right)  ^{\frac{1}{2}},\label{def Four square}%
\end{equation}
where $\mathsf{M}_{\mathbf{u}}^{s}$ is the modulation defined in Definition
\ref{def M}, and $\tau_{u_{I}}$ denotes translation by $u_{I}$.
\end{definition}

Now we are ready to define the modulated single scale trilinear conjecture,
and state the main theorem from \cite{RiSa2}.

\begin{definition}
\label{def A}Let $1<q<\infty$ and $0<\varepsilon,\nu<1$. We say the statement
$\mathcal{A}_{\operatorname*{disj}\nu}^{\operatorname*{square}}\left(
\otimes_{3}L^{\infty}\rightarrow L^{\frac{q}{3}};\varepsilon\right)  $ holds
if there is a positive constant $C_{q,\varepsilon,\nu}$ depending only on $q$,
$\varepsilon$ and $\nu$, such that,%
\begin{align}
&  \left\Vert \mathcal{S}_{\operatorname*{Fourier}}^{s,\varphi,\mathbf{u}_{1}%
}f_{1}\ \mathcal{S}_{\operatorname*{Fourier}}^{s,\varphi,\mathbf{u}_{2}}%
f_{2}\ \mathcal{S}_{\operatorname*{Fourier}}^{s,\varphi,\mathbf{u}_{3}}%
f_{3}\right\Vert _{L^{\frac{q}{3}}\left(  \mathbb{R}^{3}\right)  }\leq
C_{q,\varepsilon,\nu}2^{\varepsilon s}\left\Vert f_{1}\right\Vert _{L^{\infty
}\left(  U\right)  }\left\Vert f_{2}\right\Vert _{L^{\infty}\left(  U\right)
}\left\Vert f_{3}\right\Vert _{L^{\infty}\left(  U\right)  }%
\ ,\label{single tri Four}\\
\text{for all }s &  \in\mathbb{N}\text{ with }2^{-s}\leq\nu\text{, all }%
f_{k}\in L^{\infty}\left(  U_{k}\right)  \text{, all sequences }\mathbf{u}%
_{k}\in\mathcal{V}\text{, and all }\nu\text{-disjoint triples }\left(
U_{1},U_{2},U_{3}\right)  \subset U^{3}.\nonumber
\end{align}

\end{definition}

\begin{conjecture}
[modulated single scale Fourier square function disjoint trilinear extension
conjecture]\label{ssFsftec}For every $q>3$ there is $0<\nu<1$ such that the
statement $\mathcal{A}_{\operatorname*{disj}\nu}^{\operatorname*{square}%
}\left(  \otimes_{3}L^{\infty}\rightarrow L^{\frac{q}{3}};\varepsilon\right)
$ holds for all $0<\varepsilon<1$.
\end{conjecture}

\begin{theorem}
[C. Rios and E. Sawyer \cite{RiSa2}]\label{SFA}The strong Kakeya maximal
function conjecture \ref{Kak dual} in $\mathbb{R}^{3}$ holds \emph{if and only
if} the modulated single scale Fourier square function disjoint trilinear
extension Conjecture \ref{ssFsftec} holds for the paraboloid $\mathbb{P}^{2}$.
\end{theorem}

\subsection{Conversion to an unmodulated trilinear Fourier conjecture}

We now reformulate the trilinear Fourier extension inequality in such a way
that the role of the coefficients of the doubly smooth Alpert projections are
emphasized. For this we draw an analogy between Alpert projections and the
classical homogenous expansions of $L^{2}\left(  \mathbb{S}^{2}\right)  $
functions into homogeneous polynomials of degree $k$ on the sphere
$\mathbb{S}^{2}$.

We say that a linear combination
\[
f=\sum_{I\in\mathcal{G}_{s}\left[  U\right]  }b_{I}h_{I;\kappa}^{\eta}%
\]
of doubly smooth Alpert wavelets $h_{I;\kappa}^{\eta}$ with $I\in
\mathcal{G}_{s}\left[  U\right]  $, is an \emph{Alpert polynomial at scale
}$s$ having $\kappa$ vanishing moments. Moreover, we say that $f=\sum
_{I\in\mathcal{G}_{s}\left[  U\right]  }b_{I}h_{I;\kappa}^{\eta}$ is a
\emph{subunit} Alpert polynomial at scale $s$ if $\left\Vert f\right\Vert
_{L^{\infty}\left(  U\right)  }\leq1$.

Then we can reformulate the single scale disjoint\emph{\ }trilinear inequality
$\mathcal{E}_{\operatorname*{disj}\nu}^{\delta,\kappa}\left(  \otimes
_{3}L^{\infty}\rightarrow L^{\frac{q}{3}};\varepsilon\right)  $ in the
following equivalent form, whose proof is a simple exercise we will not give here.

\begin{lemma}
Let $1<q<\infty$ and $0<\varepsilon,\delta,\nu<1$ and $\kappa\in\mathbb{N}$.
Then the single scale disjoint\emph{\ }trilinear inequality $\mathcal{E}%
_{\operatorname*{disj}\nu}^{\delta,\kappa}\left(  \otimes_{3}L^{\infty
}\rightarrow L^{\frac{q}{3}};\varepsilon\right)  $ holds \emph{if and only if}
there is a positive constant $C_{q,\delta,\kappa,\varepsilon,\nu}$ such that%
\[
\left\Vert \mathcal{E}f_{1}\ \mathcal{E}f_{2}\ \mathcal{E}f_{3}\right\Vert
_{L^{\frac{q}{3}}\left(  B_{\mathbb{R}^{3}}\left(  0,2^{\frac{s}{1-\delta}%
}\right)  \right)  }\leq C_{q,\delta,\kappa,\varepsilon,\nu}2^{\varepsilon
s}\ ,
\]
holds for all $s\in\mathbb{N}$, $\varepsilon>0$, all subunit Alpert
polynomials $f_{k}$ of scale $s$ with $\kappa$ vanishing moments, and all
$\nu$-disjoint triples $\left(  U_{1},U_{2},U_{3}\right)  \in U^{3}$.
\end{lemma}

We now wish to reformulate the definition of the Kakeya analogue
$\mathcal{A}_{\operatorname*{disj}\nu}^{\operatorname*{square}}\left(
\otimes_{3}L^{\infty}\rightarrow L^{\frac{q}{3}};\varepsilon\right)  $ in a
corresponding way. This will prove much more problematic as we must convert
the modulated father wavelets into smooth Alpert wavelets, and it is here that
we will need to use the more general \emph{doubly} smooth Alpert
decompositions introduced below.

Heuristically, we expect that the smooth Alpert expansion of a modulation of
frequency $2^{2s}$ times a father wavelet of scale $s\in\mathbb{N}$, would
have its scales concentrated near $2s$, and this is indeed what we shall show first.

Let $t=2s\in2\mathbb{N}$ be even. Then we say that a subunit Alpert polynomial
$f=\sum_{J\in\mathcal{G}_{t}\left[  U\right]  }a_{J}h_{J;\kappa}^{\eta}$ at
scale $t=2s$ with $\kappa$ vanishing moments is of \emph{Kakeya type} if $f$
has the special form%
\begin{equation}
f\left(  x\right)  =\sum_{I\in\mathcal{G}_{s}\left[  U\right]  }b_{I}%
\sum_{J\in\mathcal{G}_{t}\left[  I\right]  }e^{ic_{J}\cdot u_{I}}h_{J;\kappa
}^{\eta}\left(  x\right)  ,\label{special}%
\end{equation}
for some $b_{I}\in\mathbb{R}$ and $u_{I}\in\mathbb{R}^{2}$, in which the
restriction of $f$ to a square $I\in\mathcal{G}_{s}\left[  U\right]  $ is a
constant times a canonical modulation sequence $\left\{  e^{ic_{J}\cdot u_{I}%
}\right\}  _{I\in\mathcal{G}_{s}\left[  U\right]  \text{ and }J\in
\mathcal{G}_{t}\left[  U\right]  }$, where $c_{J}$ is the center of the square
$J$, inner product with the sequence of doubly smooth Alpert wavelets
$\left\{  h_{J;\kappa}^{\eta}\right\}  _{J\in\mathcal{G}_{t}\left[  U\right]
} $ at scale $t$. Then we define a `martingale transform' of such a function
$f $ to be%
\[
M_{\pm}f\left(  x\right)  =\sum_{I\in\mathcal{G}_{s}\left[  U\right]  }\pm
b_{I}\sum_{J\in\mathcal{G}_{t}\left[  I\right]  }e^{ic_{J}\cdot u_{I}%
}h_{J;\kappa}^{\eta}\left(  x\right)  ,
\]
where the choice of $\mathcal{+}$ and $-$ is random, and only acts on the
squares $I$ at scale $s$.

\begin{definition}
\label{def B}Let $1<q<\infty$ and $0<\varepsilon,\nu<1$. We say the statement
$\mathcal{B}_{\operatorname*{disj}\nu}^{\operatorname*{square}}\left(
\otimes_{3}L^{\infty}\rightarrow L^{\frac{q}{3}};\varepsilon\right)  $ holds
if there is a positive constant $C_{q,\varepsilon,\nu}$ depending only on $q$,
$\varepsilon$ and $\nu$, such that,%
\begin{equation}
\mathbb{E}_{\pm,\pm,\pm}\left\Vert \mathcal{E}M_{\pm}f_{1}\ \mathcal{E}M_{\pm
}f_{2}\ \mathcal{E}M_{\pm}f_{3}\right\Vert _{L^{\frac{q}{3}}\left(
\mathbb{R}^{3}\right)  }\leq C_{q,\varepsilon,\nu}2^{\varepsilon
t}\ ,\label{expect}%
\end{equation}
holds for all $t\in2\mathbb{N}$ with $2^{-t}\leq\nu$, all subunit Alpert
polynomials at scale $t$ of Kakeya type $f_{k}$, and all $\nu$-disjoint
triples $\left(  U_{1},U_{2},U_{3}\right)  \subset U^{3}$.
\end{definition}

Here $\mathbb{E}_{\pm,\pm,\pm}$ denotes the expectation taken over independent
copies of uniformly distributed $\pm$ signs. Khintchine's inequalities show
that the probabilistic inequality (\ref{expect}) is equivalent to a square
function inequality%
\begin{equation}
\left\Vert \mathcal{S}_{\operatorname{Kakeya}}^{t,\mathbf{u}_{1}}%
f_{1}\ \mathcal{S}_{\operatorname{Kakeya}}^{t,\mathbf{u}_{2}}f_{2}%
\ \mathcal{S}_{\operatorname{Kakeya}}^{t,\mathbf{u}_{3}}f_{3}\right\Vert
_{L^{\frac{q}{3}}\left(  \mathbb{R}^{3}\right)  }\leq C_{q,\varepsilon,\nu
}2^{\varepsilon t}\ ,\label{Kak square equiv}%
\end{equation}
where the Kakeya square function $\mathcal{S}_{\operatorname{Kakeya}%
}^{t,\mathbf{u}}f $ is defined on a subunit Alpert polynomial at scale $t$ of
Kakeya type $f$ as in (\ref{special}) by%
\begin{equation}
\mathcal{S}_{\operatorname{Kakeya}}^{t,\mathbf{u}}f\equiv\left(  \sum
_{I\in\mathcal{G}_{s}\left[  U\right]  }\left\vert b_{I}\mathcal{E}\left(
\sum_{J\in\mathcal{G}_{t}\left[  I\right]  }e^{ic_{J}\cdot u_{I}}h_{J;\kappa
}^{\eta}\right)  \right\vert ^{2}\right)  ^{\frac{1}{2}},\ \ \ \ \ \text{where
}s=\frac{1}{2}t.\label{Kakeya square}%
\end{equation}
Note that the square function applies only to the `outer' layer parameterized
by $I\in\mathcal{G}_{s}\left[  U\right]  $, while the `inner' layers that are
parameterized by $J\in\mathcal{G}_{t}\left[  I\right]  $ remain linear.

Our main theorem is the equivalence of $\mathcal{A}_{\operatorname*{disj}\nu
}^{\operatorname*{square}}\left(  \otimes_{3}L^{\infty}\rightarrow L^{\frac
{q}{3}};\varepsilon\right)  $ for all $\varepsilon>0$ and $\mathcal{B}%
_{\operatorname*{disj}\nu}^{\operatorname*{square}}\left(  \otimes
_{3}L^{\infty}\rightarrow L^{\frac{q}{3}};\varepsilon\right)  $ for all
$\varepsilon>0$.

\begin{theorem}
[Conversion Theorem]The modulated condition $\mathcal{A}_{\operatorname*{disj}%
\nu}^{\operatorname*{square}}\left(  \otimes_{3}L^{\infty}\rightarrow
L^{\frac{q}{3}};\varepsilon\right)  $ in Definition \ref{def A} holds for all
$\varepsilon>0$ \emph{if and only if} the unmodulated condition $\mathcal{B}%
_{\operatorname*{disj}\nu}^{\operatorname*{square}}\left(  \otimes
_{3}L^{\infty}\rightarrow L^{\frac{q}{3}};\varepsilon\right)  $ in Definition
\ref{def B} holds for all $\varepsilon>0$.
\end{theorem}

Our proof will make liberal use of error bounds of the following form.

\begin{notation}
\label{geo Four dec}We say that a subunit Alpert polynomial $f$ at scale
$s\in\mathbb{N}$ has \emph{geometric Fourier decay}, denoted
$\operatorname*{GeoFourDec}$, if the $L^{q}$ norm of the Fourier extension
$\mathcal{E}f$ is controlled by the appropriate growth constant, i.e.%
\[
\left\Vert \mathcal{E}f\right\Vert _{L^{q}\left(  \mathbb{R}^{3}\right)  }\leq
C_{\varepsilon}2^{\varepsilon s},
\]
where the constant $C$ may depend on other relevant quantities, but is
independent of $s\in\mathbb{N}$.
\end{notation}

A typical Alpert polynomial to which we apply this notation has the form%
\[
f=\sum_{I\in\mathcal{G}_{s}\left[  U\right]  }b_{I}h_{I;\kappa}^{\eta}\ ,
\]
so that we need only apply the trivial $L^{1}\rightarrow L^{\infty}$ estimate
for the Fourier transform to $\mathcal{E}f=\sum_{I\in\mathcal{G}_{s}\left[
U\right]  }b_{I}\mathcal{E}h_{I;\kappa}^{\eta}$ to obtain%
\[
\left\Vert \mathcal{E}f\right\Vert _{L^{q}\left(  \mathbb{R}^{3}\right)  }%
\leq\sum_{I\in\mathcal{G}_{s}\left[  U\right]  }\left\vert b_{I}\right\vert
\left\Vert \mathcal{E}h_{I;\kappa}^{\eta}\right\Vert _{L^{q}\left(
\mathbb{R}^{3}\right)  }%
\]
where the nonzero coefficients $b_{I}$ are assumed to have sufficiently rapid
decay to compensate the norms $\left\Vert \mathcal{E}h_{I;\kappa}^{\eta
}\right\Vert _{L^{q}\left(  \mathbb{R}^{3}\right)  }$. Details of this sort of
estimate will appear during the course of the proof.

\subsection{A comparison of trilinear characterizations}

We can now use the Conversion Theorem to compare the Fourier extension and
Kakeya conjectures by examining the following testing inequalities that
characterize these conjectures respectively:%
\begin{align}
& \left\Vert \mathcal{E}f_{1}\ \mathcal{E}f_{2}\ \mathcal{E}f_{3}\right\Vert
_{L^{\frac{q}{3}}\left(  B_{\mathbb{R}^{3}}\left(  0,2^{\frac{s}{1-\delta}%
}\right)  \right)  }\leq C_{q,\delta,\kappa,\varepsilon,\nu}2^{\varepsilon
s},\label{Four}\\
& \text{for all subunit Alpert polynomials }f_{k}\text{ at scale }s\text{, and
for all }s\in\mathbb{N},\nonumber
\end{align}
and%
\begin{align}
& \mathbb{E}_{\pm,\pm,\pm}\left\Vert \mathcal{E}M_{\pm}f_{1}\ \mathcal{E}%
M_{\pm}f_{2}\ \mathcal{E}M_{\pm}f_{3}\right\Vert _{L^{\frac{q}{3}}\left(
\mathbb{R}^{3}\right)  }\leq C_{q,\varepsilon,\nu}2^{\varepsilon
t}\ ,\label{Kak}\\
& \text{for all subunit Alpert polynomials }f_{k}\text{ at scale }t\text{
\textbf{of Kakeya type}, and for all }t\in2\mathbb{N},\nonumber
\end{align}
where in both inequalities the appropriate assumptions on parameters
$s,t,q,\delta,\kappa,\varepsilon,\nu$ are in force.

Of course (\ref{Four}) implies (\ref{Kak}) by Khintchine's inequality and the
fact that the functions are of Kakeya type. Moreover (\ref{Kak}) is weaker
than (\ref{Four}) in \textbf{two} distinct ways, namely

\begin{enumerate}
\item the restrictions of the subunit Alpert polynomials $f$ to squares
$I\in\mathcal{G}_{s}\left[  U\right]  $ in (\ref{Kak}) have the special form
of a canonical modulation,

\item and the norms in (\ref{Kak}) are subjected to a expectation over special
martingale transforms of the subunit Alpert polynomials of Kakeya type.
\end{enumerate}

It is clear from parabolic rescaling, that if we drop the first requirement on
restrictions to modulations, then the corresponding testing condition that
assumes only the second requirement (2), is equivalent to
inequality(\ref{Four}). If instead we drop only the second requirement, then
the corresponding testing condition that assumes only the first requirement
(1), implies inequality (\ref{Kak}) and is implied by inequality (\ref{Four}).
But we do not know if it is equivalent to (\ref{Kak}).

Thus we have three trilinear characterizations of the Kakeya maximal operator
conjecture, namely the two Fourier type inequalities (\ref{single tri Four})
and (\ref{expect}), as well as the geometric/combinatoric trilinear
characterization in \cite{RiSa2}, namely that there is a positive constant
$C_{\varepsilon,\nu}$ such that\footnote{see \cite{RiSa2} for notation}%
\begin{align}
&  \left\Vert \prod_{k=1}^{3}\left(  \sum_{T_{k}\in\mathbb{T}_{k}}%
\mathbf{1}_{T_{k}}\right)  \right\Vert _{L^{\frac{1}{2}}\left(  \mathbb{R}%
^{3}\right)  }\leq C_{\varepsilon,\nu}\delta^{-\varepsilon}%
,\label{tri Kak dual}\\
&  \text{for all }\nu\text{\emph{-disjoint }families }\mathbb{T}_{k}\text{ of
}\delta\text{-separated }\delta\text{-tubes in }\mathbb{R}^{3}\text{ and
}0<\delta\leq\nu.\nonumber
\end{align}

It is unclear whether or not the two trilinear Fourier characterizations could
be any easier to establish than the trilinear geometric/combinatoric
characterization, and even whether or not the trilinear characterization
(\ref{tri Kak dual}) has a significant advantage over the bilinear
characterization proved much earlier in Tao, Vargas and Vega \cite{TaVaVe}
(apart from the fact that disjoint trilinear is implied by bilinear).
Nevertheless, the characterization using (\ref{Kak}) provides a common playing
field in which to compare the Fourier extension and Kakeya maximal operator
conjectures, and perhaps shed some light on the still open problem of the
equivalence of the Fourier extension and Kakeya strong maximal operator conjectures.

Moreover, the above comments suggest a possible approach to proving the
Fourier extension theorem for the paraboloid $\mathbb{P}^{2}$. Namely, we
decompose a general sequence of coefficients defined on $\mathcal{G}%
_{s}\left[  U\right]  $ into a sum of canonical sequences using (\ref{dec})
below, and then by `reversing' some of the arguments in the proof of the
Conversion Theorem below, we localize the Fourier extension operator applied
to these canonical sequences times smooth Alpert wavelets, using the
translation invariance of the wavelets, in the hope of making favourable norm estimates.

The aforementioned decomposition is obtained by considering the Hilbert space
$\ell^{2}\left(  \mathbb{Z}_{N}^{2}\right)  $ where $\mathbb{Z}_{N}^{2}%
\equiv\left\{  -N,-N+1,...N-1,N\right\}  ^{2}$, and then rescaling the
orthonormal basis $\left\{  \varphi^{m}\right\}  _{m\in\mathbb{Z}_{N}^{2}}$ of
canonical sequences $\varphi^{m}=\left\{  \varphi_{n}^{m}\right\}
_{n\in\mathbb{Z}_{N}^{2}}$ defined by $\varphi_{n}^{m}\equiv\frac{e^{2\pi
i\frac{n\cdot m}{2N+1}}}{2N+1}$. Indeed,
\begin{align*}
& \left\langle \varphi^{m},\varphi^{m^{\prime}}\right\rangle _{\ell^{2}\left(
\mathbb{Z}_{N}^{2}\right)  }=\sum_{n\in\mathbb{Z}_{N}^{2}}\varphi_{n}%
^{m}\varphi_{n}^{m^{\prime}}=\sum_{n\in\mathbb{Z}_{N}^{2}}\frac{e^{2\pi
in\cdot\frac{m-m^{\prime}}{2N+1}}}{\left(  2N+1\right)  ^{2}}\\
& =\frac{1}{\left(  2N+1\right)  ^{2}}\frac{\sin\left[  2\pi\left(  N+\frac
{1}{2}\right)  \left(  \frac{m_{1}-m_{1}^{\prime}}{2N+1}\right)  \right]
}{\sin\left[  \pi\left(  \frac{m_{1}-m_{1}^{\prime}}{2N+1}\right)  \right]
}\frac{\sin\left[  2\pi\left(  N+\frac{1}{2}\right)  \left(  \frac{m_{2}%
-m_{2}^{\prime}}{2N+1}\right)  \right]  }{\sin\left[  \pi\left(  \frac
{m_{2}-m_{2}^{\prime}}{2N+1}\right)  \right]  },
\end{align*}
shows that $\left\{  \varphi^{m}\right\}  _{m\in\mathbb{Z}_{N}^{2}}$ is an
orthonormal basis of $\ell^{2}\left(  \mathbb{Z}_{N}^{2}\right)  $, and thus
for $\mathbf{a}=\left\{  a_{n}\right\}  _{n\in\mathbb{Z}_{N}^{2}}\in\ell
^{2}\left(  \mathbb{Z}_{N}^{2}\right)  $, we have%
\begin{equation}
\mathbf{a}=\sum_{m\in\mathbb{Z}_{N}^{2}}\left\langle \mathbf{a},\varphi
^{m}\right\rangle \varphi^{m},\label{dec}%
\end{equation}
which represents $\mathbf{a}$ as a linear combination of canonical sequences
$\varphi^{m}$ with $\sum_{m\in\mathbb{Z}_{N}^{2}}\left\vert \left\langle
\mathbf{a},\varphi^{m}\right\rangle \right\vert ^{2}=\left\vert \mathbf{a}%
\right\vert _{\ell^{2}\left(  \mathbb{Z}_{N}^{2}\right)  }^{2}$.

\section{Supporting theorems}

In this section we work in $n\geq2$ dimensions, as the three dimensional case
is no different.

\subsection{Doubly smooth Alpert frames}

We begin with the construction from \cite{Saw7} of smooth Alpert projections
$\left\{  \bigtriangleup_{Q;\kappa}^{\eta}\right\}  _{Q\in\mathcal{D}}$ and
corresponding wavelets $\left\{  h_{Q;\kappa}^{a,\eta}\right\}  _{Q\in
\mathcal{D},\ a\in\Gamma_{n}}$ of order $\kappa$ in $n$-dimensional space
$\mathbb{R}^{n}$, provided $\eta>0$ is sufficiently small and $\kappa
\in\mathbb{N}$ is sufficiently large. First, we recall that the (unsmoothed)
Alpert wavelets $\left\{  h_{Q;\kappa}^{a}\right\}  _{a\in\Gamma}$ constructed
in \cite{RaSaWi} are an orthonormal basis for the finite dimensional vector
subspace of $L^{2}$ that consists of linear combinations of the indicators
of\ the children $\mathfrak{C}\left(  Q\right)  $ of $Q$ multiplied by
polynomials of degree at most $\kappa-1$, and such that the linear
combinations have vanishing moments on the cube $Q$ up to order $\kappa-1$:%
\[
L_{Q;k}^{2}\left(  \mu\right)  \equiv\left\{
f=\mathop{\displaystyle \sum }\limits_{Q^{\prime}\in\mathfrak{C}\left(
Q\right)  }\mathbf{1}_{Q^{\prime}}p_{Q^{\prime};k}\left(  x\right)  :\int
_{Q}f\left(  x\right)  x_{i}^{\ell}d\mu\left(  x\right)  =0,\ \ \ \text{for
}0\leq\ell\leq k-1\text{ and }1\leq i\leq n\right\}  ,
\]
where $p_{Q^{\prime};k}\left(  x\right)  =\sum_{\alpha\in\mathbb{Z}_{+}%
^{n}:\left\vert \alpha\right\vert \leq k-1\ }a_{Q^{\prime};\alpha}x^{\alpha}$
is a polynomial in $\mathbb{R}^{n}$ of degree $\left\vert \alpha\right\vert
=\alpha_{1}+...+\alpha_{n}$ at most $\kappa-1$, and $x^{\alpha}=x_{1}%
^{\alpha_{1}}x_{2}^{\alpha_{2}}...x_{n-1}^{\alpha_{n-1}}$. Let $d_{Q;\kappa
}\equiv\dim L_{Q;\kappa}^{2}\left(  \mu\right)  $ be the dimension of the
finite dimensional linear space $L_{Q;\kappa}^{2}\left(  \mu\right)  $.
Moreover, for each $a\in\Gamma_{n}$, we may assume the wavelet $h_{Q;\kappa
}^{a}$ is a translation and dilation of the unit wavelet $h_{Q_{0};\kappa}%
^{a}$, where $Q_{0}=\left[  0,1\right)  ^{n}$ is the unit cube in
$\mathbb{R}^{n}$. The Alpert projection $\bigtriangleup_{Q;\kappa}$ onto
$L_{Q;k}^{2}\left(  \mu\right)  $ is given by $\bigtriangleup_{Q;\kappa}%
f=\sum_{a\in\Gamma_{n}}\left\langle f,h_{Q;\kappa}^{a}\right\rangle
h_{Q;\kappa}^{a}$.

Given a small positive constant $\eta>0$, define a smooth approximate identity
by $\phi_{\eta}\left(  x\right)  \equiv\eta^{-n}\phi\left(  \frac{x}{\eta
}\right)  $ where $\phi\in C_{c}^{\infty}\left(  B_{\mathbb{R}^{n}}\left(
0,1\right)  \right)  $ has unit integral, $\int_{\mathbb{R}^{n}}\phi\left(
x\right)  dx=1$, and vanishing moments of \emph{positive} order less than
$\kappa$, i.e.
\begin{equation}
\int\phi\left(  x\right)  x^{\gamma}dx=\delta_{\left\vert \gamma\right\vert
}^{0}=\left\{
\begin{array}
[c]{ccc}%
1 & \text{ if } & \left\vert \gamma\right\vert =0\\
0 & \text{ if } & 0<\left\vert \gamma\right\vert <\kappa
\end{array}
\right.  .\label{van pos}%
\end{equation}
The \emph{smooth} Alpert `wavelets' were then defined in \cite{Saw7} by%
\[
h_{Q;\kappa}^{a,\eta}\equiv h_{Q;\kappa}^{a}\ast\phi_{\eta\ell\left(
Q\right)  },
\]
and we have for $0\leq\left\vert \beta\right\vert <\kappa$,%
\begin{align*}
& \int h_{Q;\kappa}^{a,\eta}\left(  x\right)  x^{\beta}dx=\int\phi_{\eta
\ell\left(  I\right)  }\ast h_{Q;\kappa}^{a}\left(  x\right)  x^{\beta}%
dx=\int\int\phi_{\eta\ell\left(  I\right)  }\left(  y\right)  h_{Q;\kappa}%
^{a}\left(  x-y\right)  x^{\beta}dx\\
& =\int\phi_{\eta\ell\left(  I\right)  }\left(  y\right)  \left\{  \int
h_{Q;\kappa}^{a}\left(  x-y\right)  x^{\beta}dx\right\}  dy=\int\phi_{\eta
\ell\left(  I\right)  }\left(  y\right)  \left\{  \int h_{Q;\kappa}^{a}\left(
x\right)  \left(  x+y\right)  ^{\beta}dx\right\}  dy\\
& =\int\phi_{\eta\ell\left(  I\right)  }\left(  y\right)  \left\{  0\right\}
dy=0,
\end{align*}
by translation invariance of Lebesgue measure.

There is a linear map $S_{\eta}^{\mathcal{D}}=S_{\kappa,\eta}^{\mathcal{D}}$,
bounded and invertible on all $L^{p}\left(  \mathbb{R}^{2}\right)  $ spaces,
$1<p<\infty$,$\,$such that if we define%
\[
\bigtriangleup_{I;\kappa}^{\eta}f\equiv\left(  \bigtriangleup_{I;\kappa
}f\right)  \ast\phi_{\eta\ell\left(  I\right)  },
\]
then%
\[
\bigtriangleup_{I;\kappa}^{\eta}f\equiv\sum_{a\in\Gamma_{n}}\left\langle
\left(  S_{\eta}^{\mathcal{D}}\right)  ^{-1}f,h_{I;\kappa}^{a}\right\rangle
h_{I;\kappa}^{a,\eta}=\sum_{a\in\Gamma_{n}}\left\langle \left(  S_{\eta
}^{\mathcal{D}}\right)  ^{-1}f,h_{I;\kappa}^{a}\right\rangle S_{\eta
}^{\mathcal{D}}h_{I;\kappa}^{a}=\sum_{a\in\Gamma_{n}}\left(  S_{\eta
}^{\mathcal{D}}\bigtriangleup_{I;\kappa}\left(  S_{\eta}^{\mathcal{D}}\right)
^{-1}\right)  f=\sum_{a\in\Gamma_{n}}\bigtriangleup_{I;\kappa}^{\spadesuit
}f\ ,
\]
where $A^{\spadesuit}$ denotes the commutator $S_{\eta}^{\mathcal{D}}A\left(
S_{\eta}^{\mathcal{D}}\right)  ^{-1}$ of an operator $A$ with $S_{\eta
}^{\mathcal{D}}$.

\begin{theorem}
[\cite{Saw7}]\label{reproducing}Let $n\geq2$ and $\kappa\in\mathbb{N}$ with
$\kappa>\frac{n}{2}$. Then there is $\eta_{0}>0$ depending on $n$ and $\kappa
$\ such that for all $0<\eta<\eta_{0}$, and for all grids $\mathcal{D}$ in
$\mathbb{R}^{n}$, and all $1<p<\infty$, there is a bounded invertible operator
$S_{\eta}^{\mathcal{D}}=S_{\kappa,\eta}^{\mathcal{D}}$ on $L^{p}$, and a
positive constant $C_{p,n,\eta}$ such that the collection of functions
$\left\{  h_{I;\kappa}^{a,\eta}\right\}  _{I\in\mathcal{D},\ a\in\Gamma_{n}}$
is a $C_{p,n,\eta}$-frame for $L^{p}$, by which we mean,%
\begin{align}
f\left(  x\right)   & =\sum_{I\in\mathcal{D}}\bigtriangleup_{I;\kappa}^{\eta
}f\left(  x\right)  ,\ \ \ \ \ \text{for all }f\in L^{p},\label{bounded below}%
\\
\text{where }\bigtriangleup_{I;\kappa}^{\eta}f  & \equiv\sum_{a\in\Gamma_{n}%
}\left\langle \left(  S_{\eta}^{\mathcal{D}}\right)  ^{-1}f,h_{I;\kappa}%
^{a}\right\rangle \ h_{I;\kappa}^{a,\eta}\ ,\nonumber
\end{align}
and with convergence of the sum in both the $L^{p}$ norm and almost
everywhere, and%
\begin{equation}
\frac{1}{C_{p,n,\eta}}\left\Vert f\right\Vert _{L^{p}}\leq\left\Vert \left(
\sum_{I\in\mathcal{D}}\left\vert \bigtriangleup_{I;\kappa}^{\eta}f\right\vert
^{2}\right)  ^{\frac{1}{2}}\right\Vert _{L^{p}}\leq C_{p,n,\eta}\left\Vert
f\right\Vert _{L^{p}},\ \ \ \ \ \text{for all }f\in L^{p}.\label{square est}%
\end{equation}
Moreover, the smooth Alpert wavelets $\left\{  h_{I;\kappa}^{a,\eta}\right\}
_{I\in\mathcal{D},\ a\in\Gamma_{n}}$ are translation and dilation invariant in
the sense that $h_{I;\kappa}^{a,\eta}$ is a translate and dilate of the mother
Alpert wavelet $h_{I_{0};\kappa}^{a,\eta}$ where $I_{0}$ is the unit cube in
$\mathbb{R}^{n}$.
\end{theorem}

\begin{notation}
\label{Notation Alpert} We will often drop the index $a$ that parameterizes
the finite set $\Gamma_{n}$ as it plays no essential role in most of what
follows, and it will be understood that when we write
\[
\bigtriangleup_{Q;\kappa}^{\eta}f=\left\langle \left(  S_{\eta}^{\mathcal{D}%
}\right)  ^{-1}f,h_{Q;\kappa}\right\rangle h_{Q;\kappa}^{\eta}=\breve
{f}\left(  Q\right)  h_{Q;\kappa}^{\eta},
\]
we \emph{actually} mean the Alpert \emph{pseudoprojection},%
\[
\bigtriangleup_{Q;\kappa}^{\eta}f=\sum_{a\in\Gamma_{n}}\left\langle \left(
S_{\eta}^{\mathcal{D}}\right)  ^{-1}f,h_{Q;\kappa}^{a}\right\rangle
h_{Q;\kappa}^{\eta,a}=\sum_{a\in\Gamma_{n}}\widehat{f_{a}}\left(  Q\right)
h_{Q;\kappa}^{a,\eta}\ ,
\]
where $\widehat{f_{a}}\left(  Q\right)  $ is a convenient abbreviation for the
inner product $\left\langle \left(  S_{\eta}^{\mathcal{D}}\right)
^{-1}f,h_{Q;\kappa}^{a}\right\rangle $ when $\kappa$ is understood. More
precisely, one can view $\widehat{f}\left(  Q\right)  =\left\{  \widehat
{f_{a}}\left(  Q\right)  \right\}  _{a\in\Gamma_{n}}$ and $h_{Q;\kappa}^{\eta
}=\left\{  h_{Q;\kappa}^{a,\eta}\right\}  _{a\in\Gamma_{n}}$ as sequences of
numbers and functions indexed by $\Gamma_{n}$, in which case $\widehat
{f}\left(  Q\right)  h_{Q;\kappa}^{\eta}$ is the dot product of these two
sequences. No confusion should arise between the Alpert coefficient
$\widehat{g}\left(  Q\right)  $, $Q\in\mathcal{G}\left[  U\right]  $ and the
Fourier transform $\widehat{g}\left(  \xi\right)  $, $\xi\in\mathbb{R}^{3}$,
as the argument in the first is a square in $\mathcal{G}\left[  U\right]  $,
while the argument in the second is a point in $\mathbb{R}^{3}$.
\end{notation}

\subsection{A dual expansion and the $SS^{\ast}$ operator}

For convenience we momentarily write $S=S_{\eta}^{\mathcal{D}}$, and recall
that%
\[
Sf\left(  x\right)  =\sum_{I\in\mathcal{D}}\left\langle f,h_{I;\kappa
}\right\rangle h_{I;\kappa}^{\eta}\left(  x\right)  ,
\]
with convergence in $L^{p}$, $1<p<\infty$, and pointwise almost everywhere. We
first claim that%
\[
S^{\ast}g\left(  y\right)  =\sum_{I\in\mathcal{D}}\left\langle g,h_{I;\kappa
}^{\eta}\right\rangle h_{I;\kappa}\left(  y\right)  ,
\]
with convergence in $L^{p}$, $1<p<\infty$, and pointwise almost everywhere.
Indeed, for all $f$ we have%
\begin{align*}
\left\langle f,S^{\ast}g\right\rangle  & =\left\langle Sf,g\right\rangle
=\int_{\mathbb{R}^{n}}Sf\left(  x\right)  g\left(  x\right)  dx=\sum
_{I\in\mathcal{D}}\int_{\mathbb{R}^{n}}\left\langle f,h_{I;\kappa
}\right\rangle h_{I;\kappa}^{\eta}\left(  x\right)  g\left(  x\right)
dx=\sum_{I\in\mathcal{D}}\left\langle f,h_{I;\kappa}\right\rangle \left\langle
g,h_{I;\kappa}^{\eta}\right\rangle \\
& =\sum_{I\in\mathcal{D}}\int_{\mathbb{R}^{n}}f\left(  y\right)  \left\langle
g,h_{I;\kappa}^{\eta}\right\rangle h_{I;\kappa}\left(  y\right)
dy=\left\langle f,\sum_{I\in\mathcal{D}}\left\langle g,h_{I;\kappa}^{\eta
}\right\rangle h_{I;\kappa}\right\rangle .
\end{align*}

As a consequence we have the dual reproducing formula,%
\[
g=S^{\ast}\left(  S^{\ast}\right)  ^{-1}g=\sum_{I\in\mathcal{D}}\left\langle
\left(  S^{\ast}\right)  ^{-1}g,h_{I;\kappa}^{\eta}\right\rangle h_{I;\kappa},
\]
with convergence in $L^{p}$, $1<p<\infty$, and pointwise almost everywhere,
and also the $SS^{\ast}$ formula,%
\[
SS^{\ast}g=S\left(  \sum_{I\in\mathcal{D}}\left\langle g,h_{I;\kappa}^{\eta
}\right\rangle h_{I;\kappa}\right)  =\sum_{I\in\mathcal{D}}\left\langle
g,h_{I;\kappa}^{\eta}\right\rangle Sh_{I;\kappa}=\sum_{I\in\mathcal{D}%
}\left\langle g,h_{I;\kappa}^{\eta}\right\rangle h_{I;\kappa}^{\eta},
\]
with convergence in $L^{p}$, $1<p<\infty$, and pointwise almost everywhere.
But $T\equiv SS^{\ast}$ is bounded and invertible, and so we have the
following $SS^{\ast}$ reproducing formula,
\[
g\left(  y\right)  =SS^{\ast}\left(  T^{-1}g\right)  \left(  y\right)
=\sum_{I\in\mathcal{D}}\left\langle T^{-1}g,h_{I;\kappa}^{\eta}\right\rangle
h_{I;\kappa}^{\eta}\left(  y\right)  ,
\]
with convergence in $L^{p}$, $1<p<\infty$, and pointwise almost everywhere.
The operator $T$ depends on the grid $\mathcal{D}$ and $\eta>0$, and we
sometimes write $T_{\eta}^{\mathcal{D}}=S_{\eta}^{\mathcal{D}}\left(  S_{\eta
}^{\mathcal{D}}\right)  ^{\ast}$ when we want to emphasize the dependence on
$\mathcal{D}$ and $\eta$. Here is the extension of Theorem \ref{reproducing}
to the bounded invertible self-adjoint linear operator $T_{\eta}^{\mathcal{D}%
}=S_{\eta}^{\mathcal{D}}\left(  S_{\eta}^{\mathcal{D}}\right)  ^{\ast}$.

\begin{notation}
In the theorem below, we have \emph{redefined} $\bigtriangleup_{I;\kappa
}^{\eta}f$ to be $\sum_{a\in\Gamma_{n}}\left\langle \left(  T_{\eta
}^{\mathcal{D}}\right)  ^{-1}f,h_{I;\kappa}^{a,\eta}\right\rangle
\ h_{I;\kappa}^{a,\eta}$, rather than $\sum_{a\in\Gamma_{n}}\left\langle
\left(  S_{\eta}^{\mathcal{D}}\right)  ^{-1}f,h_{I;\kappa}^{a}\right\rangle
\ h_{I;\kappa}^{a,\eta}$. Thus the Alpert wavelet inside the inner product is
now smooth at the expense of replacing the bounded invertible operator
$S_{\eta}^{\mathcal{D}}$ with the self-adjoint invertible operator $T_{\eta
}^{\mathcal{D}}=S_{\eta}^{\mathcal{D}}\left(  S_{\eta}^{\mathcal{D}}\right)
^{\ast}$.
\end{notation}

\begin{theorem}
\label{reproducing'}Let $n\geq2$ and $\kappa\in\mathbb{N}$ with $\kappa
>\frac{n}{2}$. Define $T_{\eta}^{\mathcal{D}}=S_{\eta}^{\mathcal{D}}\left(
S_{\eta}^{\mathcal{D}}\right)  ^{\ast}$ for $0<\eta<\eta_{0}$, where $S_{\eta
}^{\mathcal{D}}$ and $\eta_{0}$ are as in Theorem \ref{reproducing}. Then for
all $0<\eta<\eta_{0}$, and for all grids $\mathcal{D}$ in $\mathbb{R}^{n}$,
and all $1<p<\infty$, there is a positive constant $C_{p,n,\eta}$ such that
the collection of functions $\left\{  h_{I;\kappa}^{a,\eta}\right\}
_{I\in\mathcal{D},\ a\in\Gamma_{n}}$ is a $C_{p,n,\eta}$-frame for $L^{p}$, by
which we mean,%
\begin{align}
f\left(  x\right)   & =\sum_{I\in\mathcal{D}}\bigtriangleup_{I;\kappa}^{\eta
}f\left(  x\right)  ,\ \ \ \ \ \text{for all }f\in L^{p}%
,\label{bounded below'}\\
\text{where }\bigtriangleup_{I;\kappa}^{\eta}f  & \equiv\sum_{a\in\Gamma_{n}%
}\left\langle \left(  T_{\eta}^{\mathcal{D}}\right)  ^{-1}f,h_{I;\kappa
}^{a,\eta}\right\rangle \ h_{I;\kappa}^{a,\eta}\ ,\nonumber
\end{align}
and with convergence of the sum in both the $L^{p}$ norm and almost
everywhere, and%
\begin{equation}
\frac{1}{C_{p,n,\eta}}\left\Vert f\right\Vert _{L^{p}}\leq\left\Vert \left(
\sum_{I\in\mathcal{D}}\left\vert \bigtriangleup_{I;\kappa}^{\eta}f\right\vert
^{2}\right)  ^{\frac{1}{2}}\right\Vert _{L^{p}}\leq C_{p,n,\eta}\left\Vert
f\right\Vert _{L^{p}},\ \ \ \ \ \text{for all }f\in L^{p}.\label{square est'}%
\end{equation}
Moreover, the smooth Alpert wavelets $\left\{  h_{I;\kappa}^{a,\eta}\right\}
_{I\in\mathcal{D},\ a\in\Gamma_{n}}$ are translation and dilation invariant in
the sense that $h_{I;\kappa}^{a,\eta}$ is a translate and dilate of the mother
Alpert wavelet $h_{I_{0};\kappa}^{a,\eta}$ where $I_{0}$ is the unit cube in
$\mathbb{R}^{n}$.
\end{theorem}

We will use (\ref{bounded below'}) and (\ref{square est'}) going forward, and
we refer to (\ref{bounded below'}) as the doubly smooth Alpert expansion as
the smooth Alpert wavelets $h_{I;\kappa}^{a,\eta}$ appear both inside and
outside the inner product with $\left(  T_{\eta}^{\mathcal{D}}\right)  ^{-1}%
f$. Here is a brief explanation of why we need this particular decomposition
in our proof.

Fix $s\in\mathbb{N}$ and define $\mathsf{M}_{\mathbf{u}}^{s}f\left(  x\right)
\equiv e^{iu_{I}\cdot x}f\left(  x\right)  $. Consider the four reproducing
formulas for $\mathsf{M}_{\mathbf{u}}^{s}\bigtriangleup_{I;\kappa}^{\eta
}f=\mathsf{M}_{u_{I}}^{s}\bigtriangleup_{I;\kappa}^{\eta}f$,%
\begin{align*}
\mathsf{M}_{u_{I}}^{s}\bigtriangleup_{I;\kappa}^{\eta}f  & =\left\langle
f,h_{I;\kappa}\right\rangle \mathsf{M}_{u_{I}}^{s}h_{I;\kappa}^{\eta}\\
\text{\textbf{either} }  & =\left\langle f,h_{I;\kappa}\right\rangle
\sum_{J\in\mathcal{G}}\left\langle \mathsf{M}_{u_{I}}^{s}h_{I;\kappa}^{\eta
},h_{J;\kappa}\right\rangle h_{J;\kappa}\\
\text{\textbf{or} }  & =\left\langle f,h_{I;\kappa}\right\rangle \sum
_{J\in\mathcal{G}}\left\langle S^{-1}\mathsf{M}_{u_{I}}^{s}h_{I;\kappa}^{\eta
},h_{J;\kappa}\right\rangle h_{J;\kappa}^{\eta}\\
\text{\textbf{or} }  & =\left\langle f,h_{I;\kappa}\right\rangle \sum
_{J\in\mathcal{G}}\left\langle \left(  S^{\ast}\right)  ^{-1}\mathsf{M}%
_{u_{I}}^{s}h_{I;\kappa}^{\eta},h_{J;\kappa}^{\eta}\right\rangle h_{J;\kappa
}\\
\text{\textbf{or} }  & =\left\langle f,h_{I;\kappa}\right\rangle \sum
_{J\in\mathcal{G}}\left\langle T^{-1}\mathsf{M}_{u_{I}}^{s}h_{I;\kappa}^{\eta
},h_{J;\kappa}^{\eta}\right\rangle h_{J;\kappa}^{\eta}\ .
\end{align*}
Since the main theorem in \cite[Theorem 5]{Saw7} applies to square functions
associated with \emph{smooth} Alpert wavelet expansions, we must use a
reproducing formula with a smooth wavelet $h_{J;\kappa}^{\eta}$ on the far
right. This means we must consider either the expansion in line 3 or line 5.
But in order to integrate by parts, we must also have a smooth wavelet
$h_{J;\kappa}^{\eta}$ inside the inner product, and this leaves line 5 as the
only possibility. While this final expansion has only smooth wavelets after
the summation symbol $\sum$, it unfortunately also has the bounded invertible
self-adjoint operator $T^{-1}$ acting on the product of the modulation
$\mathsf{M}_{\mathbf{u}}^{s}$ and the smooth wavelet $h_{I;\kappa}^{\eta}$.
Nevertheless, we will be able to control the inner products using the error
estimates for $SS^{\ast}-\operatorname*{Id}$ developed in the next subsection.

\subsection{Localization estimates for $SS^{\ast}$}

Recall the following inner product estimates from \cite[Lemma 19]{Saw7}. For
$K\in\mathcal{G}$, define the collection of Carleson cubes of the skeleton
$\operatorname*{Skel}\left(  K\right)  \equiv\bigcup_{K\in\mathfrak{C}%
_{\mathcal{G}}\left(  J\right)  }\partial K$ of $K$ by%
\[
\operatorname*{Car}\left(  K\right)  \equiv\left\{  L\in\mathcal{G}%
:\ell\left(  L\right)  \leq\ell\left(  K\right)  \text{ and }L\cap
\operatorname*{Skel}\left(  K\right)  \neq\emptyset\right\}  .
\]
Define the $\eta$-halo of the skeleton $\operatorname*{Skel}\left(  K\right)
=\bigcup_{K\in\mathfrak{C}_{\mathcal{G}}\left(  J\right)  }\partial K$ of the
square $J$ by
\[
\mathcal{H}_{\eta}\left(  J\right)  \equiv\bigcup_{K\in\mathfrak{C}%
_{\mathcal{G}}\left(  J\right)  }\left\{  \left(  1+\eta\right)
K\setminus\left(  1-\eta\right)  K\right\}  ,
\]
where $\mathfrak{C}_{\mathcal{G}}\left(  J\right)  $ denotes the set of
$2^{n}$ dyadic children of $J$.

\begin{lemma}
[{\cite[Lemma 19]{Saw7}}]\label{inner est}Suppose $\kappa\in\mathbb{N}$ with
$\kappa>\frac{n}{2}$, $0<\eta=2^{-k}<1$, and $I,J\in\mathcal{G}$, where
$\mathcal{G}$ is a grid in $\mathbb{R}^{n}$. Then we have%
\begin{align*}
\left\vert \left\langle h_{J;\kappa},h_{J;\kappa}^{\eta}\right\rangle
\right\vert  & \approx1\text{ and }\left\vert \left\langle h_{J;\kappa}^{\eta
},h_{J^{\prime};\kappa}\right\rangle \right\vert \lesssim\eta
,\ \ \ \ \ \text{for }J\text{ and }J^{\prime}\text{ siblings},\\
\left\vert \left\langle h_{J;\kappa},h_{I;\kappa}^{\eta}\right\rangle
\right\vert  & \lesssim\eta\left(  \frac{\ell\left(  I\right)  }{\ell\left(
J\right)  }\right)  ^{\frac{n}{2}},\ \ \ \ \ \text{for }I\in
\operatorname*{Car}\left(  J\right)  ,\\
\left\vert \left\langle h_{J;\kappa},h_{I;\kappa}^{\eta}\right\rangle
\right\vert  & \lesssim\eta\left(  \frac{\ell\left(  J\right)  }{\ell\left(
I\right)  }\right)  ^{\frac{n}{2}-1},\ \ \ \ \ \text{for }J\in
\operatorname*{Car}\left(  I\right)  \text{ and }\ell\left(  J\right)
\geq\eta\ell\left(  I\right)  ,\\
\left\vert \left\langle h_{J;\kappa},h_{I;\kappa}^{\eta}\right\rangle
\right\vert  & \lesssim\frac{1}{\eta^{\kappa}}\left(  \frac{\ell\left(
J\right)  }{\ell\left(  I\right)  }\right)  ^{\kappa+\frac{n}{2}%
},\ \ \ \ \ \text{for }\ell\left(  J\right)  \leq\eta\ell\left(  I\right)
\text{ and }J\cap\mathcal{H}_{\eta}\left(  I\right)  \neq\emptyset,\\
\left\langle h_{J;\kappa},h_{I;\kappa}^{\eta}\right\rangle  &
=0,\ \ \ \ \ \text{in all other cases}.
\end{align*}

\end{lemma}

It is an easy exercise to extend the above estimates to the inner products
$\left\langle h_{J;\kappa}^{\eta},h_{I;\kappa}^{\eta}\right\rangle $ with
additional gain.

\begin{lemma}
[{\cite[Lemma 19]{Saw7}}]\label{inner est'}Suppose $\kappa\in\mathbb{N}$ with
$\kappa>\frac{n}{2}$, $0<\eta=2^{-k}<1$, and $I,J\in\mathcal{G}$, where
$\mathcal{G}$ is a grid in $\mathbb{R}^{n}$. Then we have%
\begin{align*}
\left\vert \left\langle h_{J;\kappa}^{\eta},h_{J;\kappa}^{\eta}\right\rangle
\right\vert  & \approx1\text{ and }\left\vert \left\langle h_{J;\kappa}^{\eta
},h_{J^{\prime};\kappa}^{\eta}\right\rangle \right\vert \lesssim
\eta,\ \ \ \ \ \text{for }J\text{ and }J^{\prime}\text{ siblings},\\
\left\vert \left\langle h_{J;\kappa}^{\eta},h_{I;\kappa}^{\eta}\right\rangle
\right\vert  & \lesssim\eta\left(  \frac{\ell\left(  I\right)  }{\ell\left(
J\right)  }\right)  ^{\kappa},\ \ \ \ \ \text{for }I\in\operatorname*{Car}%
\left(  J\right)  ,\\
\left\vert \left\langle h_{J;\kappa}^{\eta},h_{I;\kappa}^{\eta}\right\rangle
\right\vert  & \lesssim\eta\left(  \frac{\ell\left(  J\right)  }{\ell\left(
I\right)  }\right)  ^{\kappa},\ \ \ \ \ \text{for }J\in\operatorname*{Car}%
\left(  I\right)  \text{ and }\ell\left(  J\right)  \geq\eta\ell\left(
I\right)  ,\\
\left\vert \left\langle h_{J;\kappa}^{\eta},h_{I;\kappa}^{\eta}\right\rangle
\right\vert  & \lesssim\frac{1}{\eta^{\kappa}}\left(  \frac{\ell\left(
J\right)  }{\ell\left(  I\right)  }\right)  ^{\kappa+\frac{n}{2}%
},\ \ \ \ \ \text{for }\ell\left(  J\right)  \leq\eta\ell\left(  I\right)
\text{ and }\mathcal{H}_{\eta}\left(  J\right)  \cap\mathcal{H}_{\eta}\left(
I\right)  \neq\emptyset,\\
\left\langle h_{J;\kappa}^{\eta},h_{I;\kappa}^{\eta}\right\rangle  &
=0,\ \ \ \ \ \text{in all other cases}.
\end{align*}

\end{lemma}

Viewing $\mathcal{G}$ with its \emph{tree} structure, where an edge is placed
between each cube $I$ and its parent $\pi_{\mathcal{G}}I$, we define the `tree
distance' $\delta_{\operatorname*{tree}}\left(  J,I\right)  $ between two
dyadic cubes in $\mathcal{G}$ to be the length of the (unique) shortest path
joining $I$ and $J$ in the tree $\mathcal{G}$. Viewing $\mathcal{G}$ instead
with its \emph{graph} structure, where in addition to the edges in the tree
structure, edges are also placed between adjacent cubes in $\mathcal{G}$
having the same sidelength, we define the `graph distance' $\delta
_{\operatorname*{graph}}\left(  J,I\right)  $ between two dyadic cubes in
$\mathcal{G}$ to be the length of a shortest path joining $I$ and $J$ in the
tree $\mathcal{G}$ (shortest paths are not always unique in this graph structure).

\begin{notation}
From now on we will write $T_{\mathcal{G}}^{\eta}$ or $T_{\mathcal{G}}$ or
simply $T$ to denote the operator $T_{\eta}^{\mathcal{G}}=S_{\eta
}^{\mathcal{G}}\left(  S_{\eta}^{\mathcal{G}}\right)  ^{\ast}$ in Theorem
\ref{reproducing'} above when $\eta$ and $\mathcal{G}$ are understood from context.
\end{notation}

First we claim that the operator $T_{\mathcal{G}}^{\eta}$ is $\kappa
$\emph{-pseudolocal}, where we define an operator $L$ to be $\lambda
$\emph{-pseudolocal }if
\begin{align*}
Lh_{J;\kappa}^{\eta}  & =a_{\eta}h_{J;\kappa}^{\eta}+\sum_{J^{\prime}%
\in\mathcal{G}:\ \mathcal{H}_{\eta}\left(  J^{\prime}\right)  \cap
\mathcal{H}_{\eta}\left(  J\right)  \neq\emptyset}a_{J^{\prime}}%
^{J}h_{J^{\prime};\kappa}^{\eta}\ ,\\
\text{where }a_{\eta}  & =\left\langle h_{J_{\operatorname*{unit}};\kappa
}^{\eta},h_{J_{\operatorname*{unit}};\kappa}^{\eta}\right\rangle
=\int\left\vert h_{J_{\operatorname*{unit}};\kappa}^{\eta}\left(  x\right)
\right\vert ^{2}dx\approx1,\\
\text{and }\left\vert a_{J^{\prime}}^{J}\right\vert  & \leq\left\{
\begin{array}
[c]{ccc}%
C_{\kappa,N}2^{-\lambda\delta_{\operatorname*{graph}}\left(  J,J^{\prime
}\right)  } & \text{ if } & \mathcal{H}_{\eta}\left(  J^{\prime}\right)
\cap\mathcal{H}_{\eta}\left(  J\right)  \neq\emptyset\\
0 & \text{ if } & \mathcal{H}_{\eta}\left(  J^{\prime}\right)  \cap
\mathcal{H}_{\eta}\left(  J\right)  =\emptyset\\
0 & \text{ if } & J^{\prime}=J
\end{array}
\right.  ,
\end{align*}
where $J_{\operatorname*{unit}}=\left[  0,1\right)  ^{n}$ is the unit cube in
$\mathbb{R}^{n}$.

\begin{lemma}
Let $T_{\mathcal{G}}^{\eta}$ be the doubly smooth Alpert operator associated
with a frame $\left\{  h_{I;\kappa}^{\eta}\right\}  _{I\in\mathcal{G}}$ for
some $\kappa\geq\frac{n}{2}$. Then $T_{\mathcal{G}}^{\eta}$ is $\kappa$-pseudolocal.
\end{lemma}

\begin{proof}
To see this, we calculate%
\begin{equation}
Th_{J;\kappa}^{\eta}=\sum_{I\in\mathcal{G}}\left\langle T^{-1}\left(
Th_{J;\kappa}^{\eta}\right)  ,h_{I;\kappa}^{\eta}\right\rangle h_{I;\kappa
}^{\eta}=\sum_{I\in\mathcal{G}}\left\langle h_{J;\kappa}^{\eta},h_{I;\kappa
}^{\eta}\right\rangle h_{I;\kappa}^{\eta}=a_{\eta}h_{J;\kappa}^{\eta}%
+\sum_{I\in\mathcal{G}:\ I\neq J}\left\langle h_{J;\kappa}^{\eta},h_{I;\kappa
}^{\eta}\right\rangle h_{I;\kappa}^{\eta}\ ,\label{ident}%
\end{equation}
where%
\begin{equation}
\left\vert \left\langle h_{J;\kappa}^{\eta},h_{I;\kappa}^{\eta}\right\rangle
\right\vert \leq2^{C-\kappa\delta_{\operatorname*{graph}}\left(  I,J\right)
}\mathbf{1}_{\left\{  \mathcal{H}_{\eta}\left(  I\right)  \cap\mathcal{H}%
_{\eta}\left(  J\right)  \neq\emptyset\right\}  }\left(  I\right)
,\label{where'}%
\end{equation}
and we conclude that $T$ is $\kappa$-pseudolocal.
\end{proof}

\subsubsection{Well localized operators}

Now we generalize the definition of \emph{pseudolocal} operator to that of
\emph{well localized} operator.

\begin{definition}
Fix a frame of smooth Alpert wavelets $\left\{  h_{I;\kappa}^{\eta}\right\}
_{I\in\mathcal{G}}$. We say that an operator $L$ is $\lambda$\emph{-well
localized} if there are positive constants $C,\lambda$ such that%
\[
\left\vert \left\langle Lh_{J;\kappa}^{\eta},h_{I;\kappa}^{\eta}\right\rangle
\right\vert \leq2^{C-\lambda\delta_{\operatorname*{graph}}\left(  I,J\right)
},\ \ \ \ \ \text{for all }I,J\in\mathcal{G}.
\]

\end{definition}

Now we prove a basic result in localizing of the inverse operator $T^{-1}$.

\begin{proposition}
\label{T inv well}Let $T=T_{\mathcal{G}}^{\eta}$ be the doubly smooth Alpert
operator associated with a frame $\left\{  h_{I;\kappa}^{\eta}\right\}
_{I\in\mathcal{G}}$ for some $\kappa>n$. Then $T^{-1}$ is $\left(
\kappa-n\right)  $-well localized.
\end{proposition}

\begin{proof}
Recall that $\left\Vert \mathbb{I}-T\right\Vert _{L^{p}\rightarrow L^{p}}%
<\eta\ll1$, so that
\[
\left\Vert \mathbb{I}-a_{\eta}^{-1}T\right\Vert _{L^{p}\rightarrow L^{p}}%
\leq\left\Vert \mathbb{I}-T\right\Vert _{L^{p}\rightarrow L^{p}}+\left\vert
a_{\eta}^{-1}-1\right\vert \left\Vert T\right\Vert _{L^{p}\rightarrow L^{p}%
}<C\eta\ll1,
\]
and hence
\[
T^{-1}=\left[  a_{\eta}\mathbb{I}-\left(  a_{\eta}\mathbb{I}-T\right)
\right]  ^{-1}=a_{\eta}^{-1}\left[  \mathbb{I}-\left(  \mathbb{I}-a_{\eta
}^{-1}T\right)  \right]  ^{-1}=a_{\eta}^{-1}\sum_{n=0}^{\infty}\left(
\mathbb{I}-a_{\eta}^{-1}T\right)  ^{n}%
\]
is given by a Neumann series. Thus it suffices to show that
\begin{equation}
\left\vert \left\langle T^{-1}h_{J;\kappa}^{\eta},h_{I;\kappa}^{\eta
}\right\rangle \right\vert =\left\vert \left\langle \sum_{n=0}^{\infty}\left(
\mathbb{I}-a_{\eta}^{-1}T\right)  ^{n}h_{J;\kappa}^{\eta},h_{I;\kappa}^{\eta
}\right\rangle \right\vert \leq\sum_{n=0}^{\infty}\left\vert \left\langle
\left(  \mathbb{I}-a_{\eta}^{-1}T\right)  ^{n}h_{J;\kappa}^{\eta},h_{I;\kappa
}^{\eta}\right\rangle \right\vert \leq2^{C-\left(  \kappa-3\right)
\operatorname*{dtree}\left(  I,J\right)  },\label{suff}%
\end{equation}
for all $I,J\in\mathcal{G}$.

To this end we iterate the identity (\ref{ident}) to obtain first,%
\[
\left(  T-a_{\eta}^{-1}\mathbb{I}\right)  h_{J_{0};\kappa}^{\eta}=\sum
_{J_{1}\in\mathcal{G}:\ J_{1}\neq J_{0}}\left\langle h_{J_{0};\kappa}^{\eta
},h_{J_{1};\kappa}^{\eta}\right\rangle h_{J_{1};\kappa}^{\eta}\ ,
\]
and then%
\begin{align*}
& \left(  T-a_{\eta}^{-1}\mathbb{I}\right)  ^{2}h_{J_{0};\kappa}^{\eta}%
=\sum_{J_{1}\in\mathcal{G}:\ J_{1}\neq J_{0}}\left\langle h_{J_{0};\kappa
}^{\eta},h_{J_{1};\kappa}^{\eta}\right\rangle \left(  T-a_{\eta}%
^{-1}\mathbb{I}\right)  h_{J_{1};\kappa}^{\eta}\\
& =\sum_{J_{1}\in\mathcal{G}:\ J_{1}\neq J_{0}}\left\langle h_{J_{0};\kappa
}^{\eta},h_{J_{1};\kappa}^{\eta}\right\rangle \sum_{J_{2}\in\mathcal{G}%
:\ J_{2}\neq J_{1}}\left\langle h_{J_{1};\kappa}^{\eta},h_{J_{2};\kappa}%
^{\eta}\right\rangle h_{J_{2};\kappa}^{\eta}\\
& =\sum_{J_{2}\in\mathcal{G}}\left\{  \sum_{J_{1}\in\mathcal{G}:\ J_{0}\neq
J_{1},J_{1}\neq J_{2}}\left\langle h_{J_{0};\kappa}^{\eta},h_{J_{1};\kappa
}^{\eta}\right\rangle \left\langle h_{J_{1};\kappa}^{\eta},h_{J_{2};\kappa
}^{\eta}\right\rangle \right\}  h_{J_{2};\kappa}^{\eta}\ ,
\end{align*}
and finally in general,%
\begin{align*}
\left(  T-a_{\eta}^{-1}\mathbb{I}\right)  ^{n}h_{J_{0};\kappa}^{\eta}  &
=\sum_{J_{n}\in\mathcal{G}}a_{n}\left(  J_{0},J_{n}\right)  h_{J_{n};\kappa
}^{\eta}\ ,\\
\text{where }a_{n}\left(  J_{0},J_{n}\right)    & \equiv\sum_{\substack{J_{1}%
J_{2},...,J_{n-1}\in\mathcal{G}\\J_{k-1}\neq J_{k}for1\leq k\leq
n}}\left\langle h_{J_{0};\kappa}^{\eta},h_{J_{1};\kappa}^{\eta}\right\rangle
\left\langle h_{J_{1};\kappa}^{\eta},h_{J_{2};\kappa}^{\eta}\right\rangle
...\left\langle h_{J_{n-1};\kappa}^{\eta},h_{J_{n};\kappa}^{\eta}\right\rangle
.
\end{align*}
From (\ref{where'}) we now have the estimate,%
\begin{align}
\left\vert a_{n}\left(  J_{0},J_{n}\right)  \right\vert  & \leq\sum
_{\substack{J_{1}J_{2},...,J_{n-1}\in\mathcal{G}\\J_{k-1}\neq J_{k}for1\leq
k\leq n}}\left\vert \left\langle h_{J_{0};\kappa}^{\eta},h_{J_{1};\kappa
}^{\eta}\right\rangle \left\langle h_{J_{1};\kappa}^{\eta},h_{J_{2};\kappa
}^{\eta}\right\rangle ...\left\langle h_{J_{n-1};\kappa}^{\eta},h_{J_{n}%
;\kappa}^{\eta}\right\rangle \right\vert \label{from where'}\\
& \leq\sum_{\substack{J_{1}J_{2},...,J_{n-1}\in\mathcal{G}\\J_{k-1}\neq
J_{k}for1\leq k\leq n}}\prod_{k=1}^{n}2^{C-\kappa\delta_{\operatorname*{graph}%
}\left(  J_{k-1},J_{k}\right)  }\mathbf{1}_{\left\{  \mathcal{H}_{\eta}\left(
J_{k-1}\right)  \cap\mathcal{H}_{\eta}\left(  J_{n}\right)  \neq
\emptyset\right\}  }\left(  J_{k-1}\right)  .\nonumber
\end{align}

We next claim that there is $C>0$ such that
\begin{equation}
\sum_{n=0}^{\infty}\left\vert a_{n}\left(  J_{0},K_{0}\right)  \right\vert
\leq2^{C-\left(  \kappa-3\right)  \delta_{\operatorname*{graph}}\left(
J_{0},K_{0}\right)  },\ \ \ \ \ \text{for all }J_{0},K_{0}\in\mathcal{G}%
,\label{careful}%
\end{equation}
which proves (\ref{suff}), and shows that the operator $T^{-1}=a_{\eta}%
^{-1}\sum_{n=0}^{\infty}\left(  \mathbb{I}-a_{\eta}^{-1}T\right)  ^{n}$ is
$\left(  \kappa-3\right)  $-well localized.

The fact that $\mathcal{G}$ is a tree will play a critical role in proving
(\ref{careful}). Indeed, given $J_{k-1}$ and $J_{k}$, let $L_{k}%
\equiv\operatorname*{dtree}\left(  \left(  J_{k-1},J_{k}\right)  \right)  $
and set%
\[
\Omega\left(  J_{k-1},L_{k}\right)  \equiv\left\{  J\in\mathcal{G}%
:\delta_{\operatorname*{graph}}\left(  J_{k-1},J\right)  =L_{k}\text{ and
}\mathcal{H}_{\eta}\left(  J_{k-1}\right)  \cap\mathcal{H}_{\eta}\left(
J\right)  \neq\emptyset\right\}  .
\]
Assume first that $\ell\left(  J\right)  >\ell\left(  J_{k-1}\right)  $, say
$\ell\left(  J\right)  =2^{m}\ell\left(  J_{k-1}\right)  $. Then for each $m$
there are at most $3^{n}$ choices of $J$ for which $J\in\Omega\left(
J_{k-1},L_{k}\right)  $, i.e. the number of cubes adjacent to a cube in the
tower. Then $m\leq L_{k}$ and,%
\[
\#\left\{  J\in\Omega\left(  J_{k-1},L_{k}\right)  :\ell\left(  J\right)
>\ell\left(  J_{k-1}\right)  \right\}  \lesssim3^{n}L_{k}.
\]
If instead $\ell\left(  J\right)  \leq\ell\left(  J_{k-1}\right)  $, say
$\ell\left(  J\right)  =2^{-m}\ell\left(  J_{k-1}\right)  $, then $m\leq
L_{k}$ and
\[
\#\left\{  J\in\Omega\left(  J_{k-1},L_{k}\right)  :\ell\left(  J\right)
=2^{-m}\ell\left(  J_{k-1}\right)  \right\}  \lesssim\max\left\{  2^{n\left(
m-c\right)  },1\right\}  ,
\]
and so,%
\[
\#\left\{  J\in\Omega\left(  J_{k-1},L_{k}\right)  :\ell\left(  J\right)
\leq\ell\left(  J_{k-1}\right)  \right\}  \lesssim\sum_{m=0}^{L_{k}}%
\max\left\{  2^{n\left(  m-c\right)  },1\right\}  \lesssim\max\left\{
2^{2\left(  L_{k}-c\right)  },1\right\}  .
\]
Altogether then we have%
\[
\#\Omega\left(  J_{k-1},L_{k}\right)  \lesssim\max\left\{  2^{n\left(
L_{k}-c\right)  },3^{n}L_{k}\right\}  .
\]

On the other hand, we have the estimate%
\[
\left\vert \left\langle h_{J_{k-1};\kappa}^{\eta},h_{J;\kappa}^{\eta
}\right\rangle \right\vert \leq2^{C-\kappa L_{k}},\ \ \ \ \ \text{for }%
J\in\Omega\left(  J_{k-1},L_{k}\right)  ,
\]
and so we conclude that
\begin{align*}
\#\Omega\left(  J_{k-1},L_{k}\right)  2^{C-\kappa L_{k}}  & \leq\max\left\{
2^{n\left(  L_{k}-c\right)  },3^{n}L_{k}\right\}  2^{C-\kappa L_{k}}\\
& =\max\left\{  2^{C-nc-\left(  \kappa-n\right)  L_{k}},3^{n}2^{C-\kappa
L_{k}}L_{k}\right\}  \\
& \leq2^{C-\left(  \kappa-n\right)  L_{k}}L_{k}\leq2^{C-\left(  \kappa
-n\right)  L_{k}},
\end{align*}
with a different constant $C$.

It now follows that
\[
\left\vert a_{n}\left(  J_{0},J_{n}\right)  \right\vert \leq\prod_{k=1}%
^{n}\#\Omega\left(  J_{k-1},L_{k}\right)  2^{C-\kappa L_{k}}\leq\prod
_{k=1}^{n}2^{C-\left(  \kappa-n\right)  L_{k}}=2^{C-\left(  \kappa-n\right)
\sum_{k=1}^{n}\delta_{\operatorname*{graph}}\left(  J_{k-1},J_{k}\right)  },
\]
and hence that%
\[
\sum_{n=0}^{\infty}\left\vert a_{n}\left(  J_{0},K_{0}\right)  \right\vert
\leq\sum_{n=0}^{\infty}\sum_{\substack{J_{1},J_{2},...,J_{n-1}\in
\mathcal{G}\text{ and }J_{n}=K_{0}\\J_{k-1}\neq J_{k}\text{ for }1\leq k\leq
n}}2^{C-\left(  \kappa-n\right)  \sum_{k=1}^{n}\delta_{\operatorname*{graph}%
}\left(  J_{k-1},J_{k}\right)  }\leq2^{C-\left(  \kappa-n\right)
\operatorname*{dtree}\left(  J_{0},K_{0}\right)  }.
\]

\end{proof}

\begin{lemma}
\label{comp}If $L_{1}$ and $L_{2}$ are $\lambda$-well localized operators with
constants $C_{1}$ and $C_{2}$ respectively, then $L_{1}\circ L_{2}$ is
$\lambda$-well localized with constant $C_{1}+C_{2}$.
\end{lemma}

\begin{proof}
The proof is a straightforward exercise in manipulating the definition of a
$\lambda$-well localized operator with constant $C$.
\end{proof}

\subsubsection{An extension to strong localization}

We continue to work in $n$ dimensions. We will begin by showing that the
operator $T_{\mathcal{D}}^{\eta}$ is $\frac{n}{2}$\emph{-semilocal}, where we
define an operator $L$ to be $\lambda$\emph{-semilocal }if
\begin{align*}
Lh_{J;\kappa}  & =b_{\eta}h_{J;\kappa}^{\eta}+\sum_{J^{\prime}\in
\mathcal{G}:\ \mathcal{H}_{\eta}\left(  J^{\prime}\right)  \cap\mathcal{H}%
_{\eta}\left(  J\right)  \neq\emptyset}a_{J^{\prime}}^{J}h_{J^{\prime};\kappa
}^{\eta}\ ,\\
\text{where }b_{\eta}  & \equiv\left\langle h_{J_{\operatorname*{unit}}%
;\kappa},h_{J_{\operatorname*{unit}};\kappa}^{\eta}\right\rangle \text{ and
}\left\vert b_{\eta}\right\vert \approx1,\\
\text{and }\left\vert a_{J^{\prime}}^{J}\right\vert  & \leq\left\{
\begin{array}
[c]{ccc}%
C_{\kappa,N}2^{-\lambda\delta_{\operatorname*{graph}}\left(  J,J^{\prime
}\right)  } & \text{ if } & \mathcal{H}_{\eta}\left(  I\right)  \cap
\operatorname*{Skel}\left(  J\right)  \neq\emptyset\\
0 & \text{ if } & \mathcal{H}_{\eta}\left(  J^{\prime}\right)  \cap
\operatorname*{Skel}\left(  J\right)  =\emptyset\\
0 & \text{ if } & J^{\prime}=J
\end{array}
\right.  ,
\end{align*}

Now we can give our first semilocal result in $n$ dimensions.

\begin{lemma}
Let $T$ be the doubly smooth Alpert operator associated with a frame $\left\{
h_{I;\kappa}^{\eta}\right\}  _{I\in\mathcal{G}}$ for some $\kappa>\frac{n}{2}%
$. Then $T$ is $\frac{n}{2}$-semilocal.
\end{lemma}

\begin{proof}
To see this, we calculate%
\begin{equation}
Th_{J;\kappa}=\sum_{I\in\mathcal{G}}\left\langle h_{J;\kappa},h_{I;\kappa
}^{\eta}\right\rangle h_{I;\kappa}^{\eta}=\sum_{I\in\mathcal{G}}\left\langle
h_{J;\kappa},h_{I;\kappa}^{\eta}\right\rangle h_{I;\kappa}^{\eta}=b_{\eta
}h_{J;\kappa}^{\eta}+\sum_{I\in\mathcal{G}:\ I\neq J}\left\langle h_{J;\kappa
},h_{I;\kappa}^{\eta}\right\rangle h_{I;\kappa}^{\eta}\ ,\label{ident semi}%
\end{equation}
where by Lemma \ref{inner est},%
\begin{align}
\left\vert \left\langle h_{J;\kappa},h_{I;\kappa}^{\eta}\right\rangle
\right\vert  & \leq2^{C-\kappa\delta_{\operatorname*{graph}}\left(
I,J\right)  }\text{ if }\ell\left(  J\right)  \leq\ell\left(  I\right)
\,\text{and }\mathcal{H}_{\eta}\left(  I\right)  \cap\operatorname*{Skel}%
\left(  J\right)  \neq\emptyset\label{where' semi}\\
\left\vert \left\langle h_{J;\kappa},h_{I;\kappa}^{\eta}\right\rangle
\right\vert  & \leq2^{C-\frac{n}{2}\delta_{\operatorname*{graph}}\left(
I,J\right)  }\text{ if }\ell\left(  J\right)  >\ell\left(  I\right)
\,\text{and }\mathcal{H}_{\eta}\left(  I\right)  \cap\operatorname*{Skel}%
\left(  J\right)  \neq\emptyset,\nonumber
\end{align}
and we conclude that $T$ is $\frac{n}{2}$-pseudolocal.
\end{proof}

Now we generalize the definition of \emph{semilocal} operator to that of
\emph{well semilocalized} operator. For this we need the concept of the
hangman set $\operatorname*{Hang}\left(  I\right)  $ of a cube $I\in
\mathcal{G}$. Define
\begin{align*}
\operatorname*{Tower}\left(  I\right)    & \equiv\left\{  K\in\mathcal{G}%
:\delta_{\operatorname*{graph}}\left(  K,\left[  I,\infty\right)  \right)
\leq1\right\}  ,\\
\text{where }\left[  I,\infty\right)    & \equiv\left\{  \pi^{\left(
k\right)  }I:k\geq0\right\}
\end{align*}
to be the dyadic tower above $I$ together with its adjacent cubes, and
\begin{align*}
\operatorname*{Hang}\left(  I\right)    & \equiv\operatorname*{Tower}\left(
I\right)  \cup\operatorname*{Adj}\operatorname*{Skel}\left(  I\right)  ,\\
\text{where }\operatorname*{Adj}\operatorname*{Skel}\left(  I\right)    &
\equiv\left\{  K\in\mathcal{G}:\ell\left(  K\right)  \leq\ell\left(  I\right)
\text{ and }K\cap\operatorname*{Skel}\left(  I\right)  \neq\emptyset\right\}
\end{align*}
to be the hangman set of $I$, which resembles the stick figure in the child's
game of hangman.

\begin{definition}
Define $J_{\operatorname*{Hang}\left(  I\right)  }$ to be the (unique) cube in
$\operatorname*{Hang}\left(  I\right)  $ that is closest to $J$ in the graph
metric, i.e.
\[
\delta_{\operatorname*{graph}}\left(  J,\operatorname*{Hang}\left(  I\right)
\right)  =\min\left\{  \delta_{\operatorname*{graph}}\left(  J,K\right)
:K\in\operatorname*{Hang}\left(  I\right)  \right\}  ,
\]
sometimes referred to as the stopping time for the geodesic $\left[
J,K\right]  $ in the set $\operatorname*{Hang}\left(  I\right)  $.
\end{definition}

\begin{definition}
Fix a frame of smooth Alpert wavelets $\left\{  h_{I;\kappa}^{\eta}\right\}
_{I\in\mathcal{G}}$ in $\mathbb{R}^{n}$. We say that an operator $L$ is
$\left(  \lambda_{1},\lambda_{2}\right)  $\emph{-well semilocalized} if there
are positive constants $C,\lambda_{1},\lambda_{2}$ such that,
\[
\left\vert \left\langle Lh_{J;\kappa},h_{I;\kappa}^{\eta}\right\rangle
\right\vert \leq2^{C-\left[  \lambda_{1}\delta_{\operatorname*{graph}}\left(
J,I_{\operatorname*{Hang}\left(  J\right)  }\right)  +\lambda_{2}%
\delta_{\operatorname*{graph}}\left(  I_{\operatorname*{Hang}\left(  J\right)
},I\right)  \right]  },\ \ \ \ \ \text{for all }I,J\in\mathcal{G}.
\]

\end{definition}

There are two main differences between a \emph{well semilocalized} operator
and a \emph{well localized} operator. One is that the smooth Alpert wavelet
$h_{J;\kappa}^{\eta}$ in \emph{well localized} has here been replaced by the
nonsmooth $h_{J;\kappa}$ in \emph{well semilocalized}, and the second is that
the graph distance in \emph{well semilocalized} is weighted differently from
$I$ to $J_{\operatorname*{Hang}\left(  I\right)  }$ than it is from
$J_{\operatorname*{Hang}\left(  I\right)  }$ to $J$.

\begin{proposition}
\label{T inv well semi}Let $T$ be the doubly smooth Alpert operator associated
with a frame $\left\{  h_{I;\kappa}^{\eta}\right\}  _{I\in\mathcal{G}}$ for
some $\kappa>n$. Then $T^{-1}$ is $\left(  \frac{n}{2},\kappa-n\right)  $-well semilocalized.
\end{proposition}

\begin{proof}
We have%
\[
T^{-1}h_{J;\kappa}=T^{-2}Th_{J;\kappa}=T^{-2}\sum_{I\in\mathcal{G}%
}\left\langle h_{J;\kappa},h_{I;\kappa}^{\eta}\right\rangle h_{I;\kappa}%
^{\eta}=\sum_{I\in\mathcal{G}}\left\langle h_{J;\kappa},h_{I;\kappa}^{\eta
}\right\rangle T^{-2}h_{I;\kappa}^{\eta}%
\]
where $T^{-2}=T^{-1}\circ T^{-1}$ is $\left(  \kappa-3\right)  $-well
localized by Proposition \ref{T inv well} and Lemma \ref{comp}. Then we have%
\[
\left\langle T^{-1}h_{J;\kappa},h_{K;\kappa}^{\eta}\right\rangle =\sum
_{I\in\mathcal{G}}\left\langle h_{J;\kappa},h_{I;\kappa}^{\eta}\right\rangle
\left\langle T^{-2}h_{I;\kappa}^{\eta},h_{K;\kappa}^{\eta}\right\rangle ,
\]
and so%
\begin{align*}
& \left\vert \left\langle T^{-1}h_{J;\kappa},h_{K;\kappa}^{\eta}\right\rangle
\right\vert \leq\sum_{I\in\mathcal{G}}\left\vert \left\langle h_{J;\kappa
},h_{I;\kappa}^{\eta}\right\rangle \right\vert \left\vert \left\langle
T^{-2}h_{I;\kappa}^{\eta},h_{K;\kappa}^{\eta}\right\rangle \right\vert \\
& \leq\sum_{I\in\mathcal{G}}\left\{
\begin{array}
[c]{ccc}%
2^{C-\kappa\delta_{\operatorname*{graph}}\left(  I,J\right)  }2^{-\left(
\kappa-n\right)  \delta_{\operatorname*{graph}}\left(  I,K\right)  } & \text{
if } & \ell\left(  J\right)  \leq\ell\left(  I\right)  \,\text{and
}\mathcal{H}_{\eta}\left(  I\right)  \cap\operatorname*{Skel}\left(  J\right)
\neq\emptyset\\
2^{C-\frac{n}{2}\delta_{\operatorname*{graph}}\left(  I,J\right)  }2^{-\left(
\kappa-n\right)  \delta_{\operatorname*{graph}}\left(  I,K\right)  } & \text{
if } & \ell\left(  J\right)  >\ell\left(  I\right)  \,\text{and }%
\mathcal{H}_{\eta}\left(  I\right)  \cap\operatorname*{Skel}\left(  J\right)
\neq\emptyset
\end{array}
\right.  \\
& \leq\sum_{I\in\mathcal{G}:\ \ell\left(  J\right)  \leq\ell\left(  I\right)
\,\text{and }\mathcal{H}_{\eta}\left(  I\right)  \cap\operatorname*{Skel}%
\left(  J\right)  \neq\emptyset}2^{C-\kappa\delta_{\operatorname*{graph}%
}\left(  I,J\right)  }2^{-\left(  \kappa-n\right)  \delta
_{\operatorname*{graph}}\left(  I,K\right)  }\\
& +\sum_{I\in\mathcal{G}:\ \ell\left(  J\right)  >\ell\left(  I\right)
\,\text{and }\mathcal{H}_{\eta}\left(  I\right)  \cap\operatorname*{Skel}%
\left(  J\right)  \neq\emptyset}2^{C-\frac{n}{2}\delta_{\operatorname*{graph}%
}\left(  I,J\right)  }2^{-\left(  \kappa-n\right)  \delta
_{\operatorname*{graph}}\left(  I,K\right)  }\\
& \lesssim2^{C-\left[  \frac{n}{2}\delta_{\operatorname*{graph}}\left(
J,I_{\operatorname*{Hang}\left(  J\right)  }\right)  +\left(  \kappa-n\right)
\delta_{\operatorname*{graph}}\left(  I_{\operatorname*{Hang}\left(  J\right)
},I\right)  \right]  },
\end{align*}

upon splitting the geodesic $\left[  J,I\right]  =\left[
J,I_{\operatorname*{Hang}\left(  J\right)  }\right]  \vee\left[
I_{\operatorname*{Hang}\left(  J\right)  },I\right]  $ into the segment within
$\operatorname*{Hang}\left(  J\right)  $ and the segment outside
$\operatorname*{Hang}\left(  J\right)  $. Here $\vee$ denotes concatenation of
graph geodesics.
\end{proof}

\begin{lemma}
\label{comp semi}If $L_{1}$ and $L_{2}$ are $\left(  \lambda_{1},\lambda
_{2}\right)  $-well semilocalized operators with constants $C_{1}$ and $C_{2}$
respectively, then $L_{1}\circ L_{2}$ is $\left(  \lambda_{1},\lambda
_{2}\right)  $-well semilocalized with constant $C_{1}+C_{2}$.
\end{lemma}

\begin{proof}
The proof is a straightforward exercise in manipulating the definition of a
$\left(  \lambda_{1},\lambda_{2}\right)  $-well semilocalized operator with
constant $C$.
\end{proof}

Finally we come to the main result of this subsubsection.

\begin{theorem}
Let $T=T_{\mathcal{G}}^{\eta}$ be the doubly smooth Alpert operator associated
with a frame $\left\{  h_{I;\kappa}^{\eta}\right\}  _{I\in\mathcal{G}}$ for
some $\kappa>n$. Then $T^{-1}$ satisfies the strongly localized inequality%
\[
\left\vert \left\langle T^{-1}h_{J;\kappa},h_{I;\kappa}\right\rangle
\right\vert \leq2^{C-\left[  \lambda_{1}\left(  I,J\right)  \delta
_{\operatorname*{graph}}\left(  I,J_{\operatorname*{Hang}\left(  I\right)
}\right)  +\left(  \kappa-n\right)  \delta_{\operatorname*{graph}}\left(
J_{\operatorname*{Hang}\left(  I\right)  },I_{\operatorname*{Hang}\left(
J\right)  }\right)  +\lambda_{3}\left(  I,J\right)  \delta
_{\operatorname*{graph}}\left(  I_{\operatorname*{Hang}\left(  J\right)
},J\right)  \right]  },
\]
for all $I,J\in\mathcal{G}$ where%
\begin{align*}
\lambda_{1}\left(  I,J\right)    & =\left\{
\begin{array}
[c]{ccc}%
\kappa-n & \text{ if } & J_{\operatorname*{Hang}\left(  I\right)  }%
\in\operatorname*{Tower}\left(  I\right)  \\
\frac{n}{2} & \text{ if } & J_{\operatorname*{Hang}\left(  I\right)  }%
\in\operatorname*{Adj}\operatorname*{Skel}\left(  I\right)
\end{array}
\right.  ,\\
\lambda_{3}\left(  I,J\right)    & =\left\{
\begin{array}
[c]{ccc}%
\kappa-n & \text{ if } & I_{\operatorname*{Hang}\left(  J\right)  }%
\in\operatorname*{Tower}\left(  J\right)  \\
\frac{n}{2} & \text{ if } & I_{\operatorname*{Hang}\left(  J\right)  }%
\in\operatorname*{Adj}\operatorname*{Skel}\left(  J\right)
\end{array}
\right.  .
\end{align*}

\end{theorem}

Note that one of the differences between \emph{strongly localized} and
\emph{well semilocalized,} is that the smooth Alpert wavelet $h_{I;\kappa
}^{\eta}$ in \emph{well semilocalized} is replaced by the nonsmooth
$h_{I;\kappa}$ in \emph{strongly localized}. Another difference is that the
geodesic $\left[  I,J\right]  $ is split into three segments for the estimate;
the first in $\operatorname*{Hang}\left(  I\right)  $, the third in
$\operatorname*{Hang}\left(  J\right)  $, and the second outside of both
hangman sets. Moreover, the first and third segments are weighted by
$\lambda_{1}\left(  I,J\right)  $ and $\lambda_{3}\left(  I,J\right)  $
respectively depending on the whereabouts of the stopping times
$J_{\operatorname*{Hang}\left(  I\right)  }$ and $I_{\operatorname*{Hang}%
\left(  J\right)  }$, while the middle segment is weighted by $\kappa-n$.

\begin{proof}
From
\[
\left\vert \left\langle S^{-1}h_{J;\kappa},h_{I;\kappa}\right\rangle
\right\vert =\left\vert \left\langle S^{-1}h_{J;\kappa},S^{-1}h_{I;\kappa
}^{\eta}\right\rangle \right\vert =\left\vert \left\langle \left(  S^{\ast
}\right)  ^{-1}S^{-1}h_{J;\kappa},h_{I;\kappa}^{\eta}\right\rangle \right\vert
=\left\vert \left\langle T^{-1}h_{J;\kappa},h_{I;\kappa}^{\eta}\right\rangle
\right\vert
\]
and Proposition \ref{T inv well semi}, we have%
\begin{equation}
\left\vert \left\langle S^{-1}h_{J;\kappa},h_{I;\kappa}\right\rangle
\right\vert \leq2^{C-\left[  \frac{n}{2}\delta_{\operatorname*{graph}}\left(
I,J_{\operatorname*{Hang}\left(  I\right)  }\right)  +\left(  \kappa-n\right)
\delta_{\operatorname*{graph}}\left(  J_{\operatorname*{Hang}\left(  I\right)
},J\right)  \right]  }.\label{(i)}%
\end{equation}
We then also have%
\begin{equation}
\left\vert \left\langle \left(  S^{\ast}\right)  ^{-1}h_{J;\kappa}%
,h_{I;\kappa}\right\rangle \right\vert =\left\vert \left\langle h_{J;\kappa
},S^{-1}h_{I;\kappa}\right\rangle \right\vert \leq2^{C-\left[  \frac{n}%
{2}\delta_{\operatorname*{graph}}\left(  J,I_{\operatorname*{Hang}\left(
J\right)  }\right)  +\left(  \kappa-n\right)  \delta_{\operatorname*{graph}%
}\left(  I_{\operatorname*{Hang}\left(  J\right)  },I\right)  \right]
},\label{(ii)}%
\end{equation}
in which the roles of $I$ and $J$ are interchanged.

For a given pair of cubes $I$ and $J$ in $\mathcal{G}$, we write the graph
geodesic $\left[  I,J\right]  $ in three segments as
\[
\left[  I,J\right]  =\left[  I,J_{\operatorname*{Hang}\left(  I\right)
}\right]  \vee\left[  J_{\operatorname*{Hang}\left(  I\right)  }%
,I_{\operatorname*{Hang}\left(  J\right)  }\right]  \vee\left[
I_{\operatorname*{Hang}\left(  J\right)  },J\right]
\]
and then combining (\ref{(i)}) and (\ref{(ii)}) gives%
\[
\left\vert \left\langle \left(  S^{\ast}\right)  ^{-1}S^{-1}h_{J;\kappa
},h_{I;\kappa}\right\rangle \right\vert \leq2^{C-\left[  \frac{n}{2}%
\delta_{\operatorname*{graph}}\left(  I,J_{\operatorname*{Hang}\left(
I\right)  }\right)  +\left(  \kappa-n\right)  \delta_{\operatorname*{graph}%
}\left(  J_{\operatorname*{Hang}\left(  I\right)  },I_{\operatorname*{Hang}%
\left(  J\right)  }\right)  +\frac{n}{2}\delta_{\operatorname*{graph}}\left(
I_{\operatorname*{Hang}\left(  J\right)  },J\right)  \right]  }.
\]
which proves the theorem since $T^{-1}=\left(  S^{\ast}\right)  ^{-1}S^{-1}$.
\end{proof}

\section{Proof of the Conversion Theorem}

We start the proof the Conversion Theorem in this section, beginning with the
implication
\[
\left(  \forall\varepsilon>0\right)  \mathcal{B}_{\operatorname*{disj}\nu
}^{\operatorname*{square}}\left(  \otimes_{3}L^{\infty}\rightarrow L^{\frac
{q}{3}};\varepsilon\right)  \Longrightarrow\left(  \forall\varepsilon
>0\right)  \mathcal{A}_{\operatorname*{disj}\nu}^{\operatorname*{square}%
}\left(  \otimes_{3}L^{\infty}\rightarrow L^{\frac{q}{3}};\varepsilon\right)
.
\]
For this we first recall from \cite[Remark 25]{RiSa2}, that the dual form of
the Kakeya strong maximal operator conjecture can be reduced to proving the
case where all of the $\delta$-tubes of length $1$ are contained in the cube
$\left[  -2,2\right]  ^{n}$ (only the case $n=3$ was stated in \cite{RiSa2}
but the proof is the same for $n\geq2$). Scaling up by $2^{2s}$ we see that we
may assume the modulation vectors $u_{I}$ associated to a cube $I $ are
contained in the cube $\left[  2^{2s+1},2^{2s+1}\right]  ^{n}$. Moreover,
recalling that $U$ is a small cube near the origin in $\mathbb{R}^{n-1}$, that
the unit normal vectors $\left\{  \Phi\left(  y\right)  \right\}  _{y\in U}$
are close to $\mathbf{e}_{n}$, and that the associated tubes have vertical
height about $2^{2s}$, we may also assume that the vectors $u_{I}$ are
\emph{perpendicular} to $\mathbf{e}_{n}$ and have length bounded by
$2^{2s+10}$(by extending the length of the $2^{s}\times...\times2^{s}%
\times2^{2s}$ tubes by a factor of $9$ if necessary). Thus we may assume
\begin{equation}
u_{I}=\left(  u_{I}^{\prime},0\right)  \in\left[  2^{2s+10},2^{2s+10}\right]
^{n-1}\times\mathbb{R}.\label{may}%
\end{equation}

Viewing $\mathcal{G}$ as a tree, we define the `tree distance'
$\operatorname*{dtree}\left(  J,I\right)  $ between two dyadic cubes in
$\mathcal{G}$ to be the length of the (unique) shortest path joining $I$ and
$J$ in the tree $\mathcal{G}$. We will need the following lemma whose proof is
deferred to the next subsection.

\begin{lemma}
\label{scales}We have%
\begin{equation}
\left\vert \left\langle \mathsf{M}_{u}^{s}\bigtriangleup_{I}^{\varphi
}f,h_{J;\kappa}^{\eta}\right\rangle \right\vert \leq2^{D-\kappa
\operatorname*{dtree}\left(  J,\mathcal{G}_{2s}\left[  I\right]  \right)
},\ \ \ \ \ \text{for }I,J\in\mathcal{G},\label{scales ineq}%
\end{equation}
for some positive constant $C.$
\end{lemma}

\begin{proof}
[Proof of the Conversion Theorem ]Recall that%
\begin{align*}
\mathcal{S}_{\operatorname*{Fourier}}^{s,\varphi,\mathbf{u}}f\left(
\xi\right)   & \equiv\left(  \sum_{I\in\mathcal{G}_{s}\left[  U\right]
}\left\vert \left(  \mathsf{M}_{\mathbf{u}}^{s}\Phi_{\ast}\bigtriangleup
_{I}^{\varphi}f\right)  ^{\wedge}\left(  \xi\right)  \right\vert ^{2}\right)
^{\frac{1}{2}}=\left(  \sum_{I\in\mathcal{G}_{s}\left[  U\right]  }\left\vert
\left(  \mathsf{M}_{\mathbf{u}}^{s}\Phi_{\ast}\bigtriangleup_{I}^{\varphi
}f\right)  ^{\wedge}\left(  \xi\right)  \right\vert ^{2}\right)  ^{\frac{1}%
{2}}\\
& =\left(  \sum_{I\in\mathcal{G}_{s}\left[  U\right]  }\left\vert \left(
\Phi_{\ast}\mathsf{M}_{u^{\prime}}^{s}\bigtriangleup_{I}^{\varphi}f\right)
^{\wedge}\left(  \xi\right)  \right\vert ^{2}\right)  ^{\frac{1}{2}},
\end{align*}
where (see \cite[(4.6)]{RiSa2})%
\begin{align}
\mathsf{M}_{\mathbf{u}}^{s}\Phi_{\ast}\bigtriangleup_{I}^{\varphi}f\left(
z\right)   & =\Phi_{\ast}\mathsf{M}_{u}^{s}\bigtriangleup_{I}^{\varphi
}f\left(  z\right)  ,\label{comm}\\
\text{and }\mathsf{M}_{u}^{s}f\left(  y\right)   & \equiv\sum_{L\in
\mathcal{G}_{s}\left[  U\right]  }e^{iu_{L}\cdot y}\bigtriangleup_{L}%
^{\varphi}f\left(  y\right)  ,\nonumber
\end{align}
and the `father' wavelet psuedoprojection $\bigtriangleup_{I}^{\varphi}$ is
defined in Definition \ref{def Q}. We may assume without loss of generality
that $2^{t}\leq\left\vert u_{I}\right\vert \leq2^{t+10}$ where $t=2s$.

We first show that\ for every $q>3$, $N>0$ and $\varepsilon>0$, there is a
positive constant $C_{q,\delta,\varepsilon}$ such that
\begin{equation}
\left\Vert \left(  \sum_{I\in\mathcal{G}_{s}\left[  U\right]  }\left\vert
\mathcal{E}\mathsf{M}_{u}^{s}\bigtriangleup_{I}^{\varphi}f\right\vert
^{2}\right)  ^{\frac{1}{2}}\right\Vert _{L^{q}}\leq C_{q,\delta,\varepsilon
}2^{\varepsilon s}\left\Vert \left(  \sum_{I\in\mathcal{G}_{s}\left[
U\right]  }\left\vert \mathcal{E}\sum_{J\in\mathcal{G}_{t}^{N}\left[
I\right]  }\left\langle \mathsf{M}_{u}^{s}\bigtriangleup_{I}^{\varphi
},\bigtriangleup_{J;\kappa}^{\eta}\right\rangle \bigtriangleup_{J;\kappa
}^{\eta}f\right\vert ^{2}\right)  ^{\frac{1}{2}}\right\Vert _{L^{q}%
},\label{must show again}%
\end{equation}
for all $s\in\mathbb{N}$, all $f\in L^{\infty}\left(  U\right)  $, and all
sequences $\mathbf{u}\in\mathcal{V}_{s}^{n}$, where $2^{t-1}<\left\vert
u_{I}\right\vert \leq2^{t}$ and $t=2s$. Here
\[
\mathcal{G}_{t}^{N}\left[  I\right]  \equiv\bigcup_{m=t-N}^{t+N}%
\mathcal{G}_{m}\left[  I\right]
\]
consists of the squares $J$ with scale within distance $N$ of $t$.

Our plan of attack begins by appealing to Lemma \ref{scales} to conclude that
\[
T\mathsf{M}_{u}^{s}\bigtriangleup_{I}^{\varphi}f=\sum_{J}\left\langle
\mathsf{M}_{u}^{s}\bigtriangleup_{I}^{\varphi}f,h_{J;\kappa}^{\eta
}\right\rangle h_{J;\kappa}^{\eta}=\sum_{J\in\mathcal{G}_{s}%
:\ \operatorname*{dtree}\left(  J,\mathcal{G}_{t}\left[  I\right]  \right)
<N}\left\langle \mathsf{M}_{u}^{s}\bigtriangleup_{I}^{\varphi}f,h_{J;\kappa
}^{\eta}\right\rangle h_{J;\kappa}^{\eta}+\operatorname*{GeoFourDec},
\]
where $\operatorname*{GeoFourDec}$ is introduced and explained in Notation
\ref{geo Four dec}. Thus we have%
\[
\mathsf{M}_{u}^{s}\bigtriangleup_{I}^{\varphi}f=\sum_{J\in\mathcal{G}%
_{s}:\ \operatorname*{dtree}\left(  J,\mathcal{G}_{t}\left[  I\right]
\right)  <N}\left\langle \mathsf{M}_{u}^{s}\bigtriangleup_{I}^{\varphi
}f,h_{J;\kappa}^{\eta}\right\rangle T^{-1}h_{J;\kappa}^{\eta}%
+\operatorname*{GeoFourDec}.
\]
From Lemmas \ref{T inv well} and \ref{comp} we see that $T^{-2}$ is
$\lambda_{\kappa}$-well localized, and so we obtain%
\[
T^{-1}h_{J;\kappa}^{\eta}=\sum_{K\in\mathcal{G}}\left\langle T^{-2}%
h_{J;\kappa}^{\eta},h_{K;\kappa}^{\eta}\right\rangle h_{K;\kappa}^{\eta}%
=\sum_{K\in\mathcal{G}\text{:\ }\operatorname*{dtree}\left(  K,J\right)
<N}\left\langle T^{-2}h_{J;\kappa}^{\eta},h_{K;\kappa}^{\eta}\right\rangle
h_{K;\kappa}^{\eta}+\operatorname*{GeoFourDec}.
\]
Combining the last two displays gives%
\begin{align*}
& \mathsf{M}_{u}^{s}\bigtriangleup_{I}^{\varphi}f=\sum_{J\in\mathcal{G}%
_{s}:\ \operatorname*{dtree}\left(  J,\mathcal{G}_{t}\left[  I\right]
\right)  <N}\left\langle \mathsf{M}_{u}^{s}\bigtriangleup_{I}^{\varphi
}f,h_{J;\kappa}^{\eta}\right\rangle \left\{  \sum_{K\in\mathcal{G}%
\text{:\ }\operatorname*{dtree}\left(  K,J\right)  <N}\left\langle
T^{-2}h_{J;\kappa}^{\eta},h_{K;\kappa}^{\eta}\right\rangle h_{K;\kappa}^{\eta
}\right\}  +\operatorname*{GeoFourDec}\\
& =\sum_{K\in\mathcal{G}\text{:\ }\operatorname*{dtree}\left(  J,\mathcal{G}%
_{t}\left[  I\right]  \right)  <2N}\left\{  \sum_{J\in\mathcal{G}%
:\ \operatorname*{dtree}\left(  J,K\cup\mathcal{G}_{t}\left[  I\right]
\right)  <N}\left\langle \mathsf{M}_{u}^{s}\bigtriangleup_{I}^{\varphi
}f,h_{J;\kappa}^{\eta}\right\rangle \left\langle T^{-2}h_{J;\kappa}^{\eta
},h_{K;\kappa}^{\eta}\right\rangle \right\}  h_{K;\kappa}^{\eta}%
+\operatorname*{GeoFourDec},
\end{align*}
which shows that $\mathsf{M}_{u}^{s}\bigtriangleup_{I}^{\varphi}f$ is a sum of
smooth Alpert polynomials at scales \emph{near} $s$ and supported \emph{near}
$I$. It remains to show that these polynomials are of Kakeya type (Lemma
\ref{scales} is proved in the next subsection).

We now analyze in more detail the constant $\gamma_{K}$ depending on $K$ that
is given by the inner sum over $J$ in braces, namely%
\begin{equation}
\gamma_{K}\equiv\sum_{J\in\mathcal{G}:\ \operatorname*{dtree}\left(
J,K\cup\mathcal{G}_{s}\left[  I\right]  \right)  <N}\left\langle
\mathsf{M}_{u}^{s}\bigtriangleup_{I}^{\varphi}f,h_{J;\kappa}^{\eta
}\right\rangle \left\langle T^{-2}h_{J;\kappa}^{\eta},h_{K;\kappa}^{\eta
}\right\rangle .\label{inner sum}%
\end{equation}
Since $\mathbf{u}_{I}=\left(  u_{I},0\right)  $ is horizontal with $u_{I}%
\in\mathbb{R}^{2}$, we have $\mathbf{u}_{I}\cdot\Phi\left(  y\right)
=u_{I}\cdot y$, and we consider first the preliminary expression%
\begin{equation}
\sum_{J\in\mathcal{G}_{t}\left[  I\right]  }\left\langle \mathsf{M}%
_{\mathbf{u}}^{s,\varphi}\bigtriangleup_{I}^{\varphi}f,h_{J;\kappa}^{\eta
}\right\rangle =\sum_{J\in\mathcal{G}_{t}\left[  I\right]  }\left(  \int
e^{iu_{I}^{\prime}\cdot y}\bigtriangleup_{I}^{\varphi}f\left(  y\right)
h_{J;\kappa}^{\eta}\left(  y\right)  dy\right)  =\left\langle f,\varphi
_{I}\right\rangle \sum_{J\in\mathcal{G}_{t}\left[  I\right]  }\widehat
{\varphi_{I}h_{J;\kappa}^{\eta}}\left(  u_{I}\right)  ,\label{prelim}%
\end{equation}
where $\widehat{\varphi_{I}h_{J;\kappa}^{\eta}}$ denotes the $2$-dimensional
Fourier transform of $\varphi_{I}h_{J;\kappa}^{\eta}$.

Now using that $\mathbf{u}_{I}=\left(  u_{I}^{\prime},0\right)  $ is
horizontal and $h_{J;\kappa}^{\eta}$ is translation invariant, we have the
$2$-dimensional `equality' where $\widehat{}$ denotes the $2$-dimensional
Fourier transform,%
\begin{align}
& \widehat{\varphi_{I}h_{J;\kappa}^{\eta}}\left(  u_{I}\right)  =\int
e^{-iu_{I}\cdot y}\varphi_{I}h_{J;\kappa}^{\eta}\left(  y\right)  dy=\int
e^{-iu_{I}\cdot y}\varphi_{I}\tau_{c_{J}}h_{t;\kappa}^{\eta}\left(  y\right)
dy\label{mod prop}\\
& \sim\int e^{-iu_{I}\cdot y}\tau_{c_{J}}\left(  \varphi_{I_{s}}h_{t;\kappa
}^{\eta}\right)  \left(  y\right)  dy=\int\tau_{c_{J}} \left[  e^{-iu_{I}%
\cdot\left(  y+c_{J}\right)  }\left(  \varphi_{I_{s}}h_{t;\kappa}^{\eta
}\right)  \left(  y\right)  \right]  dy\nonumber\\
& =\int e^{-iu_{I}\cdot\left(  y+c_{J}\right)  }\left(  \varphi_{I_{s}%
}h_{t;\kappa}^{\eta}\right)  \left(  y\right)  dy=e^{-iu_{I}\cdot c_{J}}\int
e^{-iu_{I}^{\prime}\cdot y}\left(  \varphi_{I_{s}}h_{t;\kappa}^{\eta}\right)
\left(  y\right)  dy=e^{-iu_{I}\cdot c_{J}}\widehat{\varphi_{I_{s}}%
h_{t;\kappa}^{\eta}}\left(  u_{I}\right)  ,\nonumber
\end{align}
where $t=2s$, and the tilde $\sim$ is used to denote that $\varphi_{I}%
\tau_{c_{J}}h_{t;\kappa}^{\eta}=\tau_{c_{J}}\left(  \varphi_{I_{s}}%
h_{t;\kappa}^{\eta}\right)  $ only for $c_{J}\in I$. The square $I_{s}$ is the
square centered at the origin of side length $2^{-s}$, and $h_{t;\kappa}%
^{\eta}$ is the smooth Alpert wavelet $h_{I_{t};\kappa}^{\eta}$. Note that the
dependence of $\widehat{\varphi_{I}h_{J;\kappa}^{\eta}}\left(  u_{I}\right)  $
on the cube $J\in\mathcal{G}_{t}\left[  I\right]  $ has been factored out as
the exponential $e^{-iu_{I}\cdot c_{J}}$, leaving the remaining factor
$\widehat{\varphi_{I_{0}}h_{t;\kappa}^{\eta}}\left(  u_{I}\right)  $ to depend
only on the side length $2^{-t}$ of $J$. Thus the smooth Alpert polynomial
$\sum_{J\in\mathcal{G}_{t}\left[  I\right]  }\left\langle \mathsf{M}%
_{\mathbf{u}}^{s,\varphi}\bigtriangleup_{I}^{\varphi}f,h_{J;\kappa}^{\eta
}\right\rangle h_{J;\kappa}^{\eta}$ at scale $t$ is of Kakeya type.

Now we show that the general form given in (\ref{inner sum}) is a controlled
sum of smooth Alpert polynomials at scale $t$ of Kakeya type. Indeed,
\[
\gamma_{K}=\sum_{J\in\mathcal{G}:\ \operatorname*{dtree}\left(  J,K\cup
\mathcal{G}_{s}\left[  I\right]  \right)  <N}\left\langle \mathsf{M}_{u}%
^{s}\bigtriangleup_{I}^{\varphi}f,h_{J;\kappa}^{\eta}\right\rangle
\left\langle T^{-2}h_{J;\kappa}^{\eta},h_{K;\kappa}^{\eta}\right\rangle
\]
inherits the modulation property
\[
\gamma_{\tau_{c_{J}}K}=e^{-iu_{I}\cdot c_{J}}\gamma_{K}%
\]
from (\ref{mod prop}) upon using the following translation invariance property
of $T^{-2}$.

Let $\tau_{y}f\left(  x\right)  =f\left(  x-y\right)  $, and compute%
\begin{align*}
S\tau_{z}f\left(  x\right)   & =\sum_{K\in\mathcal{D}}\left\langle \tau
_{z}f,h_{K;\kappa}\right\rangle h_{K;\kappa}^{\eta}\left(  x\right)
=\sum_{K\in\mathcal{D}}\left(  \int f\left(  y-z\right)  h_{K;\kappa}\left(
y\right)  dy\right)  h_{K;\kappa}^{\eta}\left(  x\right) \\
& =\sum_{K\in\mathcal{D}}\int f\left(  y^{\prime}\right)  h_{K;\kappa}\left(
y^{\prime}+z\right)  dy^{\prime}h_{K;\kappa}^{\eta}\left(  x\right)
=\sum_{K\in\mathcal{D}}\int f\left(  y^{\prime}\right)  \left\{  h_{K;\kappa
}\left(  y^{\prime}+z\right)  h_{K;\kappa}^{\eta}\left(  x\right)  \right\}
dy^{\prime}%
\end{align*}
and%
\begin{align*}
\tau_{z}\left(  Sf\right)  \left(  x\right)   & =Sf\left(  x-z\right)
=\sum_{K\in\mathcal{D}}\left\langle f,h_{K;\kappa}\right\rangle h_{K;\kappa
}^{\eta}\left(  x-z\right)  =\sum_{K\in\mathcal{D}}\left(  \int f\left(
y^{\prime}\right)  h_{K;\kappa}\left(  y^{\prime}\right)  dy^{\prime}\right)
h_{K;\kappa}^{\eta}\left(  x-z\right) \\
& =\sum_{K\in\mathcal{D}}\int f\left(  y^{\prime}\right)  \left\{
h_{K;\kappa}\left(  y^{\prime}\right)  h_{K;\kappa}^{\eta}\left(  x-z\right)
\right\}  dy^{\prime}.
\end{align*}
Thus we have%
\begin{align*}
S\tau_{z}f\left(  x\right)  -\tau_{z}\left(  Sf\right)  \left(  x\right)   &
=\sum_{K\in\mathcal{D}}\int f\left(  y^{\prime}\right)  \left\{  h_{K;\kappa
}\left(  y^{\prime}+z\right)  h_{K;\kappa}^{\eta}\left(  x\right)
-h_{K;\kappa}\left(  y^{\prime}\right)  h_{K;\kappa}^{\eta}\left(  x-z\right)
\right\}  dy^{\prime}\\
& =\int f\left(  y^{\prime}\right)  \left\{  \sum_{K\in\mathcal{D}}\left[
h_{K-z;\kappa}\left(  y^{\prime}\right)  h_{K;\kappa}^{\eta}\left(  x\right)
-h_{K;\kappa}\left(  y^{\prime}\right)  h_{K+z;\kappa}^{\eta}\left(  x\right)
\right]  \right\}  dy^{\prime},
\end{align*}
by the translation invariance of the Alpert and smooth Alpert wavelets. Now if
we set $K^{\prime}\equiv K-z$, and take $z\in\ell\left(  K\right)
\mathbb{Z}^{n-1}$then%
\[
h_{K-z;\kappa}\left(  y^{\prime}\right)  h_{K;\kappa}^{\eta}\left(  x\right)
=h_{K^{\prime};\kappa}\left(  y^{\prime}\right)  h_{K^{\prime}+z;\kappa}%
^{\eta}\left(  x\right)  ,
\]
and so the sum in braces is%
\[
\sum_{K\in\mathcal{D}}\left[  h_{K^{\prime};\kappa}\left(  y^{\prime}\right)
h_{K^{\prime}+z;\kappa}^{\eta}\left(  x\right)  -h_{K;\kappa}\left(
y^{\prime}\right)  h_{K+z;\kappa}^{\eta}\left(  x\right)  \right]  ,
\]
in which most terms cancel at each scale, except for some of the cubes $I$ and
$I^{\prime}$ near the boundary of $U$. However, if $f$ is supported well
inside $U$, which we may assume, then $y^{\prime}$ lives well inside $U $, and
these boundary terms vanish. Thus we conclude that when $S$ acts on a
pseudoprojection of the form $F=\sum_{s=v}^{\infty}\sum_{I\in G_{s}\left[
U\right]  }\bigtriangleup_{I;\kappa}^{\eta}f$ with `frequencies' $v$ and
larger, then we have,
\[
S\tau_{z}F=\tau_{z}SF,\ \ \ \ \ \text{provided }z\in2^{-v}\mathbb{Z}^{n-1}.
\]
The same is thus true of $S^{\ast}$ and $T=SS^{\ast}$, and hence also of
$T^{-1}$, and finally then of $T^{-2}$ as well.

We remark that
\[
S\mathsf{Q}^{t}f=\sum_{K\in\mathcal{G}_{t}\left[  U\right]  }S\bigtriangleup
_{K;\kappa}f=\sum_{K\in\mathcal{G}_{t}\left[  U\right]  }\bigtriangleup
_{K;\kappa}^{\eta}f=\sum_{K\in\mathcal{G}_{t}\left[  U\right]  }\phi
_{\eta2^{-t}}\ast\bigtriangleup_{K;\kappa}f=\phi_{\eta2^{-t}}\ast
\mathsf{Q}^{t}f,
\]
so that when restricted to a projection $\mathsf{Q}^{t}f$ at level $t$, the
operator $S$ is simply convolution with $\phi_{\eta2^{-t}}$, and hence
satisfies%
\[
\tau_{z}\left(  S\mathsf{Q}^{t}f\right)  =S\tau_{z}\left(  \mathsf{Q}%
^{t}f\right)  ,\ \ \ \ \ \text{for all }z,
\]
but of course $\tau_{z}\left(  \mathsf{Q}^{t}f\right)  $ is not another
projection $\mathsf{Q}^{t}\tau_{z}f$ unless $z\in2^{-t}\mathbb{Z}^{n-1}$.
\end{proof}

\subsection{Smooth Alpert decomposition of modulated wavelets}

Here we prove Lemma \ref{scales}. Recall that the `tree distance'
$\operatorname*{dtree}\left(  J,I\right)  $ between two dyadic cubes in
$\mathcal{G}$ is the length of the (unique) shortest path joining $I$ and $J$
in the tree $\mathcal{G}$.

\begin{proof}
[Proof of Lemma \ref{scales}]Recall that%
\[
\left\langle \mathsf{M}_{u}^{s}\bigtriangleup_{I}^{\varphi}f,h_{J;\kappa
}^{\eta}\right\rangle =\left\langle f,\varphi_{I}\right\rangle \left\langle
\mathsf{M}_{u}^{s}\varphi_{I},h_{J;\kappa}^{\eta}\right\rangle =\left\langle
f,\varphi_{I}\right\rangle \int e^{iu_{I}\cdot y}\varphi_{I}\left(  y\right)
h_{J;\kappa}^{\eta}\left(  y\right)  dy,
\]
and suppose without loss of generality that $I\in\mathcal{G}_{s}\left[
U\right]  $ and%
\[
2^{2s}\leq\left\vert u_{I}\right\vert \leq2^{2s+10},
\]
so that the wavelength of the modulation $e^{iu_{I}\cdot\Phi\left(  x\right)
}$ on the cube $I$ of side length $s$ is roughly $2^{-2s}$. We may assume that
condition
\begin{equation}
\left(  1+\eta\right)  J\cap\left(  1+\eta\right)  I\neq\emptyset,\label{pd}%
\end{equation}
holds since otherwise the functions $\varphi_{I}$ and $h_{J;\kappa}^{\eta} $
have disjoint support and $\left\langle \mathsf{M}_{u}^{s}\varphi
_{I},h_{J;\kappa}^{\eta}\right\rangle =0$. Then we consider the following five cases,

(\textbf{1}): (\ref{pd}) holds and $\operatorname*{dtree}\left(
J,\mathcal{G}_{2s}\left[  I\right]  \right)  \leq D$,

(\textbf{2}): (\ref{pd}) holds and $\operatorname*{dtree}\left(
J,\mathcal{G}_{2s}\left[  I\right]  \right)  >D$ and $\ell\left(  J\right)
\leq2^{-2s}$,

(\textbf{3}): (\ref{pd}) holds and $\operatorname*{dtree}\left(
J,\mathcal{G}_{2s}\left[  I\right]  \right)  >D$ and $2^{-2s}\leq\ell\left(
J\right)  \leq2^{-s} $,

(\textbf{4}): (\ref{pd}) holds and $\operatorname*{dtree}\left(
J,\mathcal{G}_{2s}\left[  I\right]  \right)  >D$ and $2^{-s}\leq\ell\left(
J\right)  \leq1$,

(\textbf{5}): (\ref{pd}) holds and $\operatorname*{dtree}\left(
J,\mathcal{G}_{2s}\left[  I\right]  \right)  >D$ and $\ell\left(  J\right)
>1$.

In case (\textbf{1}) we simply use the trivial estimate%
\begin{align*}
\left\vert \left\langle \mathsf{M}_{u}^{s}\varphi_{I},h_{J;\kappa}^{\eta
}\right\rangle \right\vert  & \leq\int\left\vert \varphi_{I}h_{J;\kappa}%
^{\eta}\right\vert \leq\left\vert I\right\vert ^{-\frac{1}{2}}\left\vert
J\right\vert ^{-\frac{1}{2}}\left\vert \left(  1+\eta\right)  J\cap\left(
1+\eta\right)  I\right\vert \\
& \lesssim\min\left\{  \left(  \frac{\left\vert J\right\vert }{\left\vert
I\right\vert }\right)  ^{\frac{1}{2}},\left(  \frac{\left\vert I\right\vert
}{\left\vert J\right\vert }\right)  ^{\frac{1}{2}}\right\}  \leq1,
\end{align*}
which gives (\ref{scales ineq}).

In case (\textbf{2}) we have $J\subset2I$, $\ell\left(  J\right)  =2^{-r}$ and
$\operatorname*{dtree}\left(  J,\mathcal{G}_{2s}\left[  U\right]  \right)
>N$, so that $r>2s+N$. Then using the moment vanishing of $h_{J;\kappa}^{\eta
}$ we have%
\begin{align*}
\left\vert \int e^{iu_{I}\cdot y}\varphi_{I}\left(  y\right)  h_{J;\kappa
}^{\eta}\left(  y\right)  dy\right\vert  & \lesssim\left(  \frac{\ell\left(
J\right)  }{2^{-2s}}\right)  ^{\kappa}\left\Vert h_{J;\kappa}^{\eta
}\right\Vert _{L^{1}}\left\Vert \varphi_{I}\right\Vert _{L^{\infty}}\\
& \lesssim\left(  \frac{2^{-r}}{2^{-2s}}\right)  ^{\kappa}\frac{2^{-r}}%
{2^{-s}}\leq\left(  \frac{2^{-r}}{2^{-2s}}\right)  ^{\kappa+1}=C2^{-\left(
r-2s\right)  \left(  \kappa+1\right)  },
\end{align*}
which gives (\ref{scales ineq}).

In case (\textbf{3}) we have $J\subset2I$, $\ell\left(  J\right)  =2^{-r}$ and
$\operatorname*{dtree}\left(  J,\mathcal{G}_{2s}\left[  U\right]  \right)
>D$, so that $s\leq r<2s-N$. Using that the oscillatory term $e^{iu_{I}%
\cdot\Phi\left(  y\right)  }$ has wavelength $2^{-2s}$, along with the
smoothness of $h_{I;\kappa}^{\eta}\left(  y\right)  h_{J;\kappa}^{\eta}$ with
wavelength $2^{-r}$, we have for any $N\in\mathbb{N}$ that integration by
parts yields%
\begin{align*}
\left\vert \int e^{iu_{I}\cdot y}\varphi_{I}\left(  y\right)  h_{J_{1};\kappa
}^{\eta}\left(  y\right)  dy\right\vert  & \leq C_{N}\left(  \frac{2^{-2s}%
}{2^{-r}}\right)  ^{N}\left\Vert h_{J;\kappa}^{\eta}\right\Vert _{L^{1}%
}\left\Vert \varphi_{I}\right\Vert _{L^{\infty}}\\
& \leq C_{N}\left(  \frac{2^{-2s}}{2^{-r}}\right)  ^{N}\frac{\ell\left(
J\right)  }{\ell\left(  I\right)  }\lesssim\left(  \frac{2^{-2s}}{2^{-r}%
}\right)  ^{N}=2^{-\left(  r-2s\right)  N},
\end{align*}
which gives (\ref{scales ineq}) if we choose $N\geq\kappa$.

In case (\textbf{4}) we have that $J$ lies in the tower above $I$ or one of
its eight adjacent squares. Since the estimates are similar in each case we
will assume that $J$ lies in the tower above $I$ with side length at most $1$,
say $2^{-s}\leq\ell\left(  J\right)  \leq1$. Again we use that the oscillatory
term $e^{iu_{I}\cdot\Phi\left(  y\right)  }$ has wavelength $2^{-t}$, but this
time with the smoothness of $h_{I;\kappa}^{\eta}\left(  y\right)  h_{J;\kappa
}^{\eta}$ having wavelength $2^{-s}$, and so for any $N\in\mathbb{N}$,
integration by parts yields%
\begin{align*}
& \left\vert \int e^{iu_{I}\cdot y}\varphi_{I}\left(  y\right)  h_{J;\kappa
}^{\eta}\left(  y\right)  dy\right\vert \leq C_{N}\left(  \frac{2^{-2s}%
}{2^{-s}}\right)  ^{N}\left\Vert h_{J_{2};\kappa}^{\eta}\right\Vert _{L^{1}%
}\left\Vert \varphi_{I}\right\Vert _{L^{\infty}}\\
& \leq C_{N}\left(  \frac{2^{-2s}}{2^{-s}}\right)  ^{N}\frac{\ell\left(
J\right)  }{\ell\left(  I\right)  }\lesssim\left(  \frac{2^{-2s}}{2^{-s}%
}\right)  ^{N}=2^{-sN}\leq2^{-\frac{1}{2}\operatorname*{dtree}\left(
J,\mathcal{G}_{t}\left[  I\right]  \right)  N}=2^{-\kappa\operatorname*{dtree}%
\left(  J,\mathcal{G}_{t}\left[  I\right]  \right)  },
\end{align*}
if we choose $N\geq2\kappa$.

In case (\textbf{5}) we can control the corresponding term involving squares
larger than $U$, using the argument that was used to prove an analogous result
in Lemma 2 in \cite{Saw7}, and we do not repeat the argument here.
\end{proof}

\subsection{Passage from multiscale to singe scale and completion of the
proof}

To pass from the multiscale inequality (\ref{must show again}) to the desired
single scale inequality (\ref{expect}), we will use the Kakeya square function
equivalence (\ref{Kak square equiv}) together with the Bourgain Guth
pigeonholing argument as was done in \cite{RiSa2}. One of the key points in
the proof is a parabolic rescaling argument from \cite{TaVaVe}, that requires
parabolic invariance of an appropriate family of functions (which in
\cite{TaVaVe} and \cite{BoGu} was simply the family of functions in
$L^{\infty}$ bounded by $1$). The appropriate family of functions here is the
collection of subunit smooth Alpert polynomials of Kakeya type, which is
easily seen to be parabolically invariant. The details of the proof then
proceed as in \cite{RiSa2} are not repeated here.

Now we can complete our proof of the forward implication (\ref{for}) of the
Conversion Theorem. Indeed, we write%
\[
\pm\mathsf{Q}_{\mathbf{u}}^{s}f_{k}=\pm\mathsf{M}_{\mathbf{u}}^{s}%
f_{k}+\operatorname*{GeoFourDec}=M_{\pm}F_{k}+\operatorname*{GeoFourDec}%
\]
for an appropriate subunit Alpert polynomial $F_{k}$ at scale $s$. Then from
Khintchine's inequality we obtain that the left hand side of
(\ref{single tri Four}) satisfies%
\begin{align*}
& \left\Vert \mathcal{S}_{\operatorname*{Fourier}}^{s,\varphi,\mathbf{u}_{1}%
}f_{1}\ \mathcal{S}_{\operatorname*{Fourier}}^{s,\varphi,\mathbf{u}_{2}}%
f_{2}\ \mathcal{S}_{\operatorname*{Fourier}}^{s,\varphi,\mathbf{u}_{3}}%
f_{3}\right\Vert _{L^{\frac{q}{3}}\left(  \mathbb{R}^{3}\right)  }\\
& \approx\mathbb{E}_{\pm,\pm,\pm}\left\Vert \left(  \pm\mathsf{M}_{\mathbf{u}%
}^{s}\Phi_{\ast}\bigtriangleup_{I_{1}}^{\varphi}f_{1}\right)  ^{\wedge
}\ \left(  \pm\mathsf{M}_{\mathbf{u}}^{s}\Phi_{\ast}\bigtriangleup_{I_{2}%
}^{\varphi}f_{2}\right)  ^{\wedge}\ \left(  \pm\mathsf{M}_{\mathbf{u}}^{s}%
\Phi_{\ast}\bigtriangleup_{I_{3}}^{\varphi}f_{3}\right)  ^{\wedge}\right\Vert
_{L^{\frac{q}{3}}\left(  \mathbb{R}^{3}\right)  }\\
& \lesssim\mathbb{E}_{\pm,\pm,\pm}\left\Vert \left(  \mathcal{S}%
_{\operatorname{Kakeya}}^{t,\mathbf{u}}f_{1}\right)  \ \left(  \mathcal{S}%
_{\operatorname{Kakeya}}^{t,\mathbf{u}}f_{2}\right)  \ \left(  \mathcal{S}%
_{\operatorname{Kakeya}}^{t,\mathbf{u}}f_{3}\right)  \right\Vert _{L^{\frac
{q}{3}}\left(  \mathbb{R}^{3}\right)  }+\left\Vert \operatorname*{GeoFourDec}%
\right\Vert _{L^{\frac{q}{3}}\left(  \mathbb{R}^{3}\right)  }\lesssim
2^{\varepsilon s},
\end{align*}
by (\ref{expect}), (\ref{Kak square equiv}), and where the Kakeya square
function $\mathcal{S}_{\operatorname{Kakeya}}^{t,\mathbf{u}}$ is defined in
(\ref{Kakeya square}). Now we sum over the scales $s$ in question which only
enlarges the small constant $\varepsilon$ by a fixed factor.

At this point we have completed the proof of the forward implication
\begin{equation}
\left(  \forall\varepsilon>0\right)  \mathcal{B}_{\operatorname*{disj}\nu
}^{\operatorname*{square}}\left(  \otimes_{3}L^{\infty}\rightarrow L^{\frac
{q}{3}};\varepsilon\right)  \Longrightarrow\left(  \forall\varepsilon
>0\right)  \mathcal{A}_{\operatorname*{disj}\nu}^{\operatorname*{square}%
}\left(  \otimes_{3}L^{\infty}\rightarrow L^{\frac{q}{3}};\varepsilon\right)
.\label{for}%
\end{equation}
Turning to the converse implication, we claim that the modulated inequality
(\ref{single tri Four}) implies the unmodulated inequality (\ref{expect})
using estimates similar to those used above, together with the equivalence of
expectation estimates and square function estimates that follows from the
Khintchine inequality. Details of the proof of this converse implication are
left to the reader.

\end{document}